\documentclass[reqno,12pt]{amsart} 
\usepackage{vmargin}
\setpapersize{A4}

\usepackage{amsmath} 

\usepackage{amsfonts}   

\usepackage{amssymb}      

\usepackage{eufrak}      

\usepackage{amscd}      

\usepackage{amsthm}      
   
\usepackage{epsfig}      

\usepackage{amstext}      

\usepackage[all,line,dvips]{xy} 
\CompileMatrices 

\newcommand{\Ad}{\operatorname{Ad}}

\newcommand{\id}{\operatorname{id}} 
\newcommand{\Aut}{\operatorname{Aut}}

 \newcommand{\Ext}{\operatorname{Ext}} 
\newcommand{\diag}{\operatorname{diag}}

\newcommand{\Span}{\operatorname{Span}}
\newcommand{\Tr}{\operatorname{Tr}}

 \newcommand{\supp}{\operatorname{supp}}

\newcommand{\Det}{\operatorname{Det}}

\newcommand{\alg}{\operatorname{alg}}

\newcommand{\ev}{\operatorname{ev}}
\newcommand{\Sec}{\operatorname{Sec}}
\newcommand{\sign}{\operatorname{sign}}
\newcommand{\pt}{\operatorname{point}}
\newcommand{\cof}{\operatorname{cof}}

   \theoremstyle{plain}
   \newtheorem{thm}{Theorem}[section]
   \newtheorem{prop}[thm]{Proposition}
   \newtheorem{lemma}[thm]{Lemma}  
   \newtheorem{cor}[thm]{Corollary}
   \theoremstyle{definition}
   
   \newtheorem{defn}[thm]{Definition}
   \newtheorem{example}[thm]{Example}
   \theoremstyle{remark}
   
   \newtheorem{remark}[thm]{Remark}

\usepackage{graphicx}

   \numberwithin{equation}{section}

        \date{\today}

\title[$C^*$-algebras and affine transformations]{The $C^*$-algebra
  of an affine map on the 3-torus}
  
\author{Kasper K.S. Andersen and Klaus Thomsen}

\DeclareMathOperator\coker{coker}

\date{\today}

\address{Institut for Matematiske Fag, University of Copenhagen,
  DK-2100, Copenhagen, Denmark}

\email{kksa@math.ku.dk}

\address{Institut for matematik, Ny Munkegade, 8000 Aarhus C, Denmark}

\email{matkt@imf.au.dk}

\begin{document}

\maketitle

\begin{abstract} We study the $C^*$-algebra of
  an affine map on a
  compact abelian group and give necessary and sufficient conditions
  for strongly transitivity when the group is a torus. The structure
  of the $C^*$-algebra is completely determined for all strongly
  transitive affine maps on a torus of dimension one, two or three. 
 
\end{abstract}

\section{Introduction}

The purpose with the present paper is to
present a complete description and classification of the simple
$C^*$-algebras which arise from the generalised crossed product
construction of Renault, Deaconu and Anantharaman-Delaroche when it is
applied to affine maps on tori of dimension $\leq 3$. The paper is written in
the conviction that progress on our understanding of the
relationship between dynamical systems and operator algebras can
benefit both areas and that it is improved by having rich classes of
examples where the dynamical systems and the associated
$C^*$-algebras are equally tractable. The affine maps of tori constitute
a class
of dynamical systems that are well studied and whose
structures are relatively transparent when compared to other
systems. As the present paper will demonstrate our knowledge of
$C^*$-algebras is now comprehensive enough to allow a complete
identification and classification of the corresponding $C^*$-algebras,
provided the affine maps are strongly transitive and the
dimension of the torus does not exceed 3.

There are many other compact abelian groups for which it would be
desirable to have a better understanding of the $C^*$-algebras 
associated to the locally homeomorphic affine maps. For this reason we maintain a high
level of generality before we specialise to tori of low
dimension. Specifically we first describe the general
construction of Renault, Deaconu and Anantharaman-Delaroche from
\cite{Re}, \cite{De} and \cite{An}, which produces a locally compact
\'etale groupoid and hence a $C^*$-algebra out of a local
homeomorphism. We develop a bit of the structure
theory that we need when we specialise to affine maps. In particular, we show that the $KK$-equivalence
class of the $C^*$-algebra is preserved by an appropriate notion of
homotopy. It follows from this that the $C^*$-algebra of a locally
homeomorphic and surjective affine
map on a path-connected compact group is KK-equivalent, in a unit-preserving way, to the $C^*$-algebra
of its linear part. This means that as far as the calculation of the
K-groups is concerned it suffices to consider group
endomorphisms. Furthermore, it follows from the classification
theorem of Kirchberg and Phillips that the $C^*$-algebras of two strongly
transitive affine maps on the same compact metrizable path-connected group are isomorphic provided they are purely
infinite and the maps have the same linear parts. Thanks to a recent
result from \cite{CT} we know that the $C^*$-algebra of a locally
injective surjection on a compact finite dimensional metric space is purely infinite if it is simple, provided only that
the map is not injective. All in all this means that for the
$C^*$-algebras of non-injective strongly
transitive locally homeomorphic affine maps on a compact metrizable path-connected
group it is not only the K-theory, but
also the algebra itself which is completely determined by the K-theory of the $C^*$-algebra coming from the
linear part of the map. Before we apply this to tori we first show
that the $C^*$-algebra of a locally homeomorphic affine
map on a compact abelian group is the
universal $C^*$-algebra generated by a unitary representation of the
dual group and an isometry subject to two relations, cf. Theorem
\ref{universal}. This result is motivated by a recent paper by Cuntz
and Vershik (\cite{CV}) where this is done for exact endomorphisms.

Turning the attention to tori we first give necessary and sufficient conditions for an affine surjection
on a torus to be strongly transitive, cf. Theorem \ref{OK?2}. In
remains then to calculate the $K$-theory, and it is in order to obtain a complete
calculation, covering all strongly transitive locally homeomorphic
affine maps, that we restrict to tori of dimension $\leq 3$. Thus an
important part of the paper is devoted to
the calculation of the K-theory groups of $C^*$-algebras coming from
surjective group endomorphisms of tori. When all the eigenvalues of the integer matrix which
defines the endomorphism are strictly larger than one in absolute
value the endomorphism is expanding, and in this case the calculation was
performed in dimension 1 and 2 in \cite{EHR}. We use here the same
six-terms exact sequence which was obtained in \cite{EHR}, but in
order to
avoid the assumption in \cite{EHR} about the absolute value of the
eigenvalues, we take a more
elementary approach which exploits that we can play with all integer
matrices, and we complete in this way the calculation for general
integer matrices in dimension 2 and
3. The main tool is the K\"unneth theorem, and we avoid using the description of the $K$-theory of a
torus as an exterior algebra. It is not at all clear that our bare-hands
methods can be made to work in higher dimensions, but we
believe that the simplicity of the approach has some virtues in itself.

It remains then to handle the injective case which means that we must
determine the $C^*$-algebras arising from minimal affine homeomorphisms. On the circle these are
just the irrational rotation algebras and they are well understood. On
the two-torus the minimal affine homeomorphisms are all conjugate to
one of the Furstenberg-transformations whose $C^*$-algebras have been 
characterised through the work of Lin and Phillips, \cite{LP},
\cite{Ph2}. We show that both the methods and the results of Lin and
Phillips carry over with little effort to the three-dimensional
case. This part of the paper has some overlap with recent work of
Reihani, \cite{Rei}, where the K-theory of the $C^*$-algebras of
Furstenberg transformations is studied.

Finally, we summarise our results in the three Sections
\ref{n=1x}, \ref{alg-two} and \ref{alg-3}. They contain a
description of the ordered $K$-theory groups, together with the
position of the distinguished element of the $K_0$-group represented
by the unit, for all the simple $C^*$-algebras arising from strongly
transitive affine maps on tori of dimension
$\leq 3$. This characterises these
$C^*$-algebras since they are all classifiable by $K$-theory.

\section{Algebras from local homeomorphisms}

In this section we describe the construction of an \'etale groupoid
and a $C^*$-algebra from a local
homeomorphism. It was introduced in increasing generality by
J. Renault \cite{Re}, V. Deaconu \cite{De} and Anantharaman-Delaroche
\cite{An}. Although the focus in this paper is on cases where the space is
compact it will be crucial to have access to statements and results
from the locally compact case.

\subsection{The definition} 

Let $X$ be a second countable locally compact Hausdorff space and
$\varphi : X \to X$ a local homeomorphism. Set
$$
\Gamma_{\varphi} = \left\{ (x,k,y) \in X \times \mathbb Z  \times X :
  \ \exists n,m \in \mathbb N, \ k = n -m , \ \varphi^n(x) =
  \varphi^m(y)\right\} .
$$
This is a groupoid with the set of composable pairs being
$$
\Gamma_{\varphi}^{(2)} \ =  \ \left\{\left((x,k,y), (x',k',y')\right) \in \Gamma_{\varphi} \times
  \Gamma_{\varphi} : \ y = x'\right\}.
$$
The multiplication and inversion are given by 
$$
(x,k,y)(y,k',y') = (x,k+k',y') \ \text{and}  \ (x,k,y)^{-1} = (y,-k,x)
.
$$
Note that the unit space of $\Gamma_{\varphi}$ can be identified with
$X$ via the map $x \mapsto (x,0,x)$. Under this identification the
range map $r: \Gamma_{\varphi} \to X$ is the projection $r(x,k,y) = x$
and the source map the projection $s(x,k,y) = y$.

To turn $\Gamma_{\varphi}$ into a locally compact topological groupoid, fix $k \in \mathbb Z$. For each $n \in \mathbb N$ such that
$n+k \geq 0$, set
$$
{\Gamma_{\varphi}}(k,n) = \left\{ \left(x,l, y\right) \in X \times \mathbb
  Z \times X: \ l =k, \ \varphi^{k+n}(x) = \varphi^n(y) \right\} .
$$
This is a closed subset of the topological product $X \times \mathbb Z
\times X$ and hence a locally compact Hausdorff space in the relative
topology.
Since $\varphi$ is locally injective $\Gamma_{\varphi}(k,n)$ is an open subset of
$\Gamma_{\varphi}(k,n+1)$ and hence the union
$$
{\Gamma_{\varphi}}(k) = \bigcup_{n \geq -k} {\Gamma_{\varphi}}(k,n) 
$$
is a locally compact Hausdorff space in the inductive limit topology. The disjoint union
$$
\Gamma_{\varphi} = \bigcup_{k \in \mathbb Z} {\Gamma_{\varphi}}(k)
$$
is then a locally compact Hausdorff space in the topology where each
${\Gamma_{\varphi}}(k)$ is an open and closed set. In fact, as is easily verified, $\Gamma_{\varphi}$ is a locally
compact groupoid in the sense of \cite{Re} and an \'etale groupoid,
i.e. the range and source maps are local homeomorphisms.

To obtain a
$C^*$-algebra, consider the space $C_c\left(\Gamma_{\varphi}\right)$ of
continuous compactly supported functions on $\Gamma_{\varphi}$. They form
a $*$-algebra with respect to the convolution-like product
$$
f  g (x,k,y) = \sum_{z,n+ m = k} f(x,n,z)g(z,m,y)
$$
and the involution
$$
f^*(x,k,y) = \overline{f(y,-k,x)} .
$$
To obtain a $C^*$-algebra, let $x \in X$ and consider the
Hilbert space $H_x$ of square summable functions on $s^{-1}(x) = \left\{ (x',k,y')
    \in \Gamma_{\varphi} : \ y' = x \right\}$ which carries a
  representation $\pi_x$ of the $*$-algebra $C_c\left(\Gamma_{\varphi}\right)$
  defined such that
\begin{equation*}\label{pirep}
\left(\pi_x(f)\psi\right)(x',k, x) = \sum_{z, n+m = k}
f(x',n,z)\psi(z,m,x)  
\end{equation*}
when $\psi  \in H_x$. One can then define a
$C^*$-algebra $C^*_r\left(\Gamma_{\varphi}\right)$ as the completion
of $C_c\left(\Gamma_{\varphi}\right)$ with respect to the norm
$$
\left\|f\right\| = \sup_{x \in X} \left\|\pi_x(f)\right\| .
$$
Since we assume that $X$ is second countable it follows that
$C^*_r\left(\Gamma_{\varphi}\right)$ is separable. It is this $C^*$-algebra we study in the present paper when $\varphi$ is an
affine map.

Note that the $C^*$-algebra can be constructed from any locally compact \'etale groupoid in the place of
$\Gamma_{\varphi}$, see e.g. \cite{Re}, \cite{An}. Note also that $C^*_r\left(\Gamma_{\varphi}\right)$
is nothing but the classical crossed product $C_0(X) \times_{\varphi}
\mathbb Z$ when $\varphi$ is a homeomorphism.


\subsection{The structure}

By construction $C^*_r\left(\Gamma_{\varphi}\right)$ carries an action
$\beta$ by the circle group $\mathbb T$ defined such that
$$
\beta_{\lambda}(f)(x,k,y) = \lambda^k f(x,k,y)
$$
when $f \in C_c\left(\Gamma_{\varphi}\right)$. This is the \emph{gauge
  action} and it gives us an important tool for the study of the
structure of $C^*_r\left(\Gamma_{\varphi}\right)$. To describe the
fixed point algebra of the gauge action note that the canonical
conditional expectation $P : C^*_r\left(\Gamma_{\varphi}\right) \to
C^*_r\left(\Gamma_{\varphi}\right)^{\beta}$, given by
$$
P(a) = \int_{\mathbb T} \beta_{\lambda}(a) \ d\lambda,
$$
maps $C_c\left(\Gamma_{\varphi}\right)$ onto
$C_c\left(\Gamma_{\varphi}(0)\right)$. If we denote the open
subgroupoid $\Gamma_{\varphi}(0)$ by $R_{\varphi}$, it follows that
$$  
C^*_r\left(\Gamma_{\varphi}\right)^{\beta} =
C^*_r\left(R_{\varphi}\right ) .
$$
To unravel the structure of $C^*_r\left(R_{\varphi}\right )$ and $C^*_r\left(\Gamma_{\varphi}\right)$,
consider for $n \in \mathbb N$
the set
 $$
R(\varphi^n) = \left\{ (x,y) \in X \times X : \ \varphi^n(x) =
  \varphi^n(y)\right\} .
$$ 
Since $\varphi$ is a local homeomorphism $R\left(\varphi^n\right)$ is
a locally compact \'etale groupoid (an equivalence relation, in fact)
in the relative topology inherited from $X \times X$,
and we can consider its (reduced) groupoid $C^*$-algebra
$C^*_r\left(R\left(\varphi^n\right)\right)$. Now $R(\varphi^n)$ can be
identified with an open subgroupoid of $R_{\varphi} \subseteq
\Gamma_{\varphi}$ via the map $(x,y) \mapsto (x,0,y)$ and when we
suppress this identification in the notation we have that
$$
R_{\varphi} = \bigcup_n R\left(\varphi^n\right) .
$$  
It follows that the embeddings
$C_c\left(R\left(\varphi^n\right)\right) \subseteq
C_c\left(R\left(\varphi^{n+1}\right)\right) \subseteq
C_c\left(R_{\varphi}\right)$ extend to embeddings
$C^*_r\left(R\left(\varphi^n\right)\right) \subseteq
C^*_r\left(R\left(\varphi^{n+1}\right)\right) \subseteq
C_r^*\left(R_{\varphi}\right)$, cf. e.g. Proposition 1.9 in
\cite{Ph3}, and hence that
\begin{equation}\label{union}
C^*_r\left(R_{\varphi}\right ) = \overline{\bigcup_n
  C^*_r\left(R\left(\varphi^n\right)\right)} .
\end{equation}

\begin{lemma}\label{morita} $C^*_r(R(\varphi))$ is Morita equivalent to
  $C_0(\varphi(X))$.
\end{lemma}
\begin{proof}
Let
$$
G_{\varphi} = \left\{(y,x) \in X \times X : \ y = \varphi(x) \right\}
$$
be the graph of $f$. When $h \in C_c(G_{\varphi})$ and $f \in
C_c\left(R(\varphi)\right)$ define $hf : G_{\varphi} \to \mathbb C$ such that
$$
hf(y,x) = \sum_{z \in \varphi^{-1}(y)} h(y,z)f(z,x) .
$$
Then $hf \in C_c(G_{\varphi})$ and we have turned $C_c(G_{\varphi})$ into a right
$C_c(R(\varphi))$-module. Similarly, when $g \in C_c(\varphi(X))$ be define $gh \in
C_c(G_{\varphi})$ such that
$$
gh(y,x) = g(y)h(y,x),
$$
so that $C_c(G_{\varphi})$ is also a left $C_c(\varphi(X))$-module. Define a $C_c\left(R(\varphi)\right)$-valued 'inner products' on $C_c(G_{\varphi})$ such that
$$
\left<h,k\right>(x,y) = \overline{h(\varphi(x),x)}k(\varphi(y),y)
$$
and a $C_c(\varphi(X))$-valued 'inner product' such that
$$
\left( h,k\right)(y) = \sum_{z \in \varphi^{-1}(y)} h(y,z) \overline{k(y,z)} .
$$
In this way $C_c(G_{\varphi})$ becomes a
$C_c(\varphi(X))$-$C_c(R(\varphi))$-pre-imprimitivity bimodule as defined by Raeburn
and Williams in Definition 3.9 of \cite{RW} and then Proposition 3.12
of \cite{RW} shows that the completion of this bimodule is the
required $C_0(\varphi(X))$-$C^*_r\left(R(\varphi)\right)$-imprimitivity bimodule.
\end{proof}

Let $\mathbb K$ denote the $C^*$-algebra of compact operators on a
separable, infinite dimensional Hilbert space.

By applying Lemma \ref{morita} to $\varphi^n$ and combining with a well-known result of Brown, Green
and Rieffel, \cite{BGR}, we conclude that
$$
C^*_r\left(R\left(\varphi^n\right)\right) \otimes \mathbb K \simeq
C_0\left(\varphi^n(X)\right) \otimes \mathbb K.
$$
In particular, it follows from (\ref{union}) that $C^*_r(R_{\varphi})$ is an inductive limit of
$C^*$-algebras stably isomorphic to abelian $C^*$-algebras. When $X$ is compact and $\varphi$ is
  surjective it follows that $C^*_r\left(R_{\varphi}\right)$ is the
  inductive limit of a unital sequence of homogeneous $C^*$-algebras
  with spectrum $X$.

The next step will be to show that the gauge action is full.

\begin{lemma}\label{fullspectral} Elements of the form $f  g^*$,
  where $f,g \in C_c\left(\Gamma_{\varphi}(1)\right)$, span a dense
  subspace in $C^*_r\left(R_{\varphi}\right) =
  C^*_r\left(\Gamma_{\varphi}\right)^{\beta}$, and the same is true
  for the elements of the form $h k^*$ where $h,k \in
  C_c\left(\Gamma_{\varphi}(-1)\right)$.
\end{lemma}
\begin{proof} For each $n \in \mathbb N$ set
$$
R\left(\varphi^n\right) = \left\{ (x,0,y) \in R_{\varphi} : \
  \varphi^n(x) = \varphi^n(y) \right\} .
$$
Let $F \in C_c\left(R\left(\varphi^n\right)\right), \ n
  \geq 2$. Using a
    partition of unity we can write $F$ as a sum of functions in
    $C_c\left(R\left(\varphi^n\right)\right)$ each of
    which is supported in a subset of $R\left(\varphi^n\right)$ of the
    form $R\left(\varphi^n\right) \cap (U \times V)$ where $U,V$ are
    open subsets of $X$ where $\varphi^{n+1}$ is injective. We assume
    therefore that $F$ is supported in $R\left(\varphi^n\right) \cap
    (U \times V)$. Set $U_0 = r\left( R\left(\varphi^n\right) \cap (U
      \times V)\right)$ and $V_0 = s\left( R\left(\varphi^n\right) \cap (U
      \times V)\right)$, both open subsets of $X$. Set $K =
    r\left(\supp F\right)$, a compact subset of $U_0$. Let $h
    \in C_c(X)$ be such that $\supp h \subseteq U_0$ and $h(x) = 1$
    for all $x \in K$. Set
$$
A = \Gamma_{\varphi}(1,n) \cap \left( U_0 \times \{1\} \times
  \varphi(U_0)\right)
$$
and
$$
B = \Gamma_{\varphi}(-1,n) \cap \left( \varphi(U_0) \times \{-1\} \times
  V_0\right)
$$
which are open in $\Gamma_{\varphi}(1)$ and $\Gamma_{\varphi}(-1)$, respectively.
For every $(x,1,y) \in A$, set $f(x,1,y) = h(x)$ and note that $f$ has
compact support in $A$. When $(x,-1,y) \in B$ there is a unique
element $x' \in U_0$ such that $\varphi(x') = x$ and $(x',y) \in
R\left(\varphi^n\right)$. We can therefore define $g : B \to \mathbb R$ such that
$g(x,-1,y) = F(x',y)$. Extending
$f$ and $g$ to be zero outside $A$ and $B$, respectively, we can
consider them as elements of $C_c(\Gamma_{\varphi})$. Then $f, g^* \in
C_c\left(\Gamma_{\varphi}(1) \right)$. Since $f  g
= f  \left(g^*\right)^* = F$ this completes the proof of the first assertion because
$\bigcup_n C_c\left(R\left( \varphi^n\right)\right)$ is dense in
$C^*_r\left(R_{\varphi}\right)$. The second assertion is
proved in the same way.
\end{proof}

\begin{thm}\label{kishitakai} There is an automorphism $\alpha$ on
  $C^*_r\left(R_{\varphi}\right) \otimes \mathbb K$ such that
  $C^*_r\left(\Gamma_{\varphi}\right) \otimes \mathbb K$ is
  $*$-isomorphic to the crossed product $\left(C^*_r\left(R_{\varphi}\right)
  \otimes \mathbb K\right) \rtimes_{\alpha} \mathbb Z$.
\begin{proof} It follows from Lemma \ref{fullspectral} that Theorem 2
  of \cite{KT} applies to give an isomorphism
\begin{equation}\label{tensoreq}
\left(C^*_r\left(\Gamma_{\varphi}\right) \rtimes_{\beta} \mathbb
  T\right) \otimes \mathbb K \ \simeq \ C^*_r\left(R_{\varphi}\right)
\otimes \mathbb K.
\end{equation}
Let $\alpha_0$ be the automorphism of $C^*_r\left(\Gamma_{\varphi}\right) \rtimes_{\beta} \mathbb
  T$ generating the action dual to $\beta$. Then $\left(C^*_r\left(\Gamma_{\varphi}\right) \rtimes_{\beta} \mathbb
  T\right) \rtimes_{\alpha_0} \mathbb Z  \simeq
C^*_r\left(\Gamma_{\varphi}\right) \otimes \mathbb K$ by Takai
duality, cf. e.g. Theorem 7.9.3 of \cite{Pe}. Thus when we let
$\alpha$ be the automorphism of $ C^*_r\left(R_{\varphi}\right)
\otimes \mathbb K$ corresponding to $\alpha_0 \otimes \id_{\mathbb K}$
under the isomorphism (\ref{tensoreq}) we deduce that $\left(C^*_r\left(R_{\varphi}\right)
  \otimes \mathbb K\right) \rtimes_{\alpha} \mathbb Z \simeq C^*_r\left(\Gamma_{\varphi}\right) \otimes \mathbb K$.
\end{proof}
\end{thm}

When $\varphi$ is proper and surjective, we can realise
$C^*_r\left(\Gamma_{\varphi}\right)$ as a crossed product by an
endomorphism via the procedure described in \cite{De} and \cite{An}, and this can be
used to give an alternative proof of Theorem
\ref{kishitakai}. Without properness such an approach seems impossible.

\begin{cor}\label{RScat} 
\begin{enumerate} 
\item[a)] $C^*_r\left(\Gamma_{\varphi}\right)$ is a
  separable nuclear $C^*$-algebra in the boot-strap category of
  Rosenberg and Schochet, \cite{RS}.
\item[b)] Assume that $\varphi$ is surjective and that $C_0(X)$ is
$KK$-contractible. It follows that
$C^*_r\left(\Gamma_{\varphi}\right)$ is $KK$-contractible.
\end{enumerate}
\end{cor}
\begin{proof} a) is an immediate consequence of the preceding and b)
  follows from a) since Theorem \ref{kishitakai} and the
  Pimsner-Voiculescu
  exact sequence, \cite{PV}, implies that the $K$-groups of
  $C^*_r\left(\Gamma_{\varphi}\right)$ are both zero when $C_0(X)$ is $KK$-contractible.
\end{proof}

\subsection{The Deaconu-Muhly six terms exact sequence}\label{deamuh}

Generalising a result of Deaconu and Muhly from \cite{DM} it was shown
in \cite{Th2} that there is a six terms exact sequence which can be
used to calculate the K-theory groups of
$C^*_r\left(\Gamma_{\varphi}\right)$. To describe it observe that
$$
\Gamma_{\varphi}(1,0) = \left\{ (x,1,y) \in \Gamma_{\varphi} : \ y =
  \varphi(x) \right\} 
$$ 
is closed and open in $\Gamma_{\varphi}$. As in \cite{Th2} we denote by $E$ the closure of
  $C_c\left(\Gamma_{\varphi}(1,0)\right)$ in
  $C^*_r\left(\Gamma_{\varphi}\right)$. Then $E$ is a Hilbert $C_0(X)$-bimodule with 'inner product' $(f,g) =
f^*g$ and the natural bi-module structure arising from the observation
that $C_0(X)EC_0(X) \subseteq E$. Thus $E$ is a $C^*$-correspondence over
$C_0(X)$ in the sense of \cite{Ka} and the inclusions $C_0(X)
\subseteq C^*_r\left(\Gamma_{\varphi}\right)$ and $E \subseteq
C^*_r\left(\Gamma_{\varphi}\right)$ give rise to a $*$-homomorphism
$\rho : \mathcal O_E \to C^*_r\left(\Gamma_{\varphi}\right)$, where
$\mathcal O_E$ is the $C^*$-algebra introduced and studied by Katsura
in \cite{Ka}, cf. Definition 3.5 in \cite{Ka}.  It was shown in Proposition 3.2 of \cite{Th2} $\rho$ is a
$*$-isomorphism, leading to the following theorem, cf. Theorem 8.6 in \cite{Ka}.

\begin{thm}\label{6terms} Let $[E] \in KK\left(C_0(X),C_0(X)\right)$
  be the element represented by the embedding $C_0(X) \subseteq \mathbb
  K(E)$ and let $\iota : C_0(X) \to
  C^*_r\left(\Gamma_{\varphi}\right)$ be the canonical embedding. There is an exact sequence
\begin{equation*}
\begin{xymatrix}{
K_0\left(C_0(X)\right) \ar[r]^-{\id - [E]_*} &  K_0\left(C_0(X)\right)
\ar[r]^-{\iota_*} & K_0\left(C^*_r\left(\Gamma_{\varphi}\right)\right) \ar[d]
  \\
 K_1\left(C^*_r\left(\Gamma_{\varphi}\right)\right) \ar[u] &
   K_1\left(C_0(X)\right) \ar[l]^{\iota_*} & K_1\left(C_0(X)\right) \ar[l]^-{\id - [E]_*}}
\end{xymatrix}
\end{equation*}
\end{thm}

\subsection{Homotopy of local homeomorphisms}\label{homotopy}

Let $Y$ be a compact metric space. A path $\sigma_t : Y \to
Y, t \in [0,1]$, of surjective local homeomorphisms is called \emph{a homotopy of
  local homeomorphisms} when the map
$\Sigma : [0,1] \times Y \to [0,1] \times Y$ defined by
\begin{equation}\label{KK2}
\Sigma (t,y) = \left(t, \sigma_t(y)\right) 
\end{equation}
is a local
homeomorphism. We say then that $\{\sigma_t\}$ is a homotopy of local
homeomorphism connecting $\sigma_0$ and $\sigma_1$, and that $\sigma_0$
and $\sigma_1$ are \emph{homotopic as local homeomorphisms}.

\begin{lemma}\label{KKlem} Let $\sigma_0 : Y \to Y$ and $\sigma_1: Y \to Y$
  be surjective local homeomorphisms. Assume that $\sigma_0$ and $\sigma_1$ are homotopic
  as local homeomorphisms. It follows that there is a KK-equivalence
  $\lambda \in
  KK\left(C^*_r\left(\Gamma_{\sigma_0}\right),C^*_r\left(\Gamma_{\sigma_1}\right)\right)$ such that the induced isomorphism $\lambda_* : K_0\left( C^*_r\left(\Gamma_{\sigma_0}\right)\right) \to K_0\left(C^*_r\left(\Gamma_{\sigma_1}\right)\right)$ takes the element represented by the unit in $C^*_r\left(\Gamma_{\sigma_0}\right)$ to the one represented by the unit in $C^*_r\left(\Gamma_{\sigma_1}\right)$.
\begin{proof} Consider a homotopy $\left\{\sigma_t\right\}$ of local
  homeomorphisms connecting $\sigma_0$ to $\sigma_1$. Define $\Sigma : [0,1] \times Y \to [0,1] \times Y$ by
  (\ref{KK2}) and observe that $\{0\} \times Y$ and $\{1\} \times Y$
are both closed totally $\Sigma$-invariant subsets of $[0,1] \times
Y$. By Proposition 4.6 of \cite{CT} we have therefore surjective
$*$-homomorphisms $\pi_i : C^*_r\left(\Gamma_{\Sigma}\right) \to
C^*_r\left(\Gamma_{\sigma_i}\right)$ such that
$$ 
\ker \pi_i \simeq C^*_r\left(\Gamma_{\Sigma|_{Z_i}}\right),
$$
where 
$$
Z_i = \left([0,1] \backslash \left\{i\right\}\right) \times
      Y ,
$$
$i = 0,1$. Since $C_0(Z_i)$ is a 
contractible $C^*$-algebra it follows from Corollary \ref{RScat} that
$C^*_r\left(\Gamma_{\Sigma|_{Z_i}}\right)$ is $KK$-contractible. 

Let $\bullet$ denote the Kasparov
product. Since we deal with separable nuclear $C^*$-algebras it
follows from Theorem 19.5.7 of \cite{Bl} that
\begin{equation}\label{eq3} 
\begin{xymatrix}{
KK\left(C^*_r\left(\Gamma_{\sigma_i}\right) ,
  C^*_r\left(\Gamma_{\Sigma}\right)\right) \ar[rr]^-{ x \mapsto
  \left[\pi_i\right] \bullet x} & & KK\left(C^*_r\left(\Gamma_{\sigma_i}\right) ,
  C^*_r\left(\Gamma_{\sigma_i}\right)\right)
}\end{xymatrix}
\end{equation}
and
\begin{equation}\label{eq4} 
\begin{xymatrix}{
KK\left(C^*_r\left(\Gamma_{\Sigma}\right) ,
  C^*_r\left(\Gamma_{\Sigma}\right)\right) \ar[rr]^-{ x \mapsto
 \left[\pi_i\right] \bullet x} & & KK\left(C^*_r\left(\Gamma_{\Sigma}\right) ,
  C^*_r\left(\Gamma_{\sigma_i}\right)\right)
}\end{xymatrix}
\end{equation}
are both isomorphisms because $\ker \pi_i$ is KK-contractible. It
follows from the surjectivity of (\ref{eq3}) that there is an element
$\left[\pi_i\right]^{-1} \in KK\left(C^*_r\left(\Gamma_{\sigma_i}\right) ,
  C^*_r\left(\Gamma_{\Sigma}\right)\right)$ such that
$\left[\pi_i\right] \bullet \left[\pi_i\right]^{-1} =
\left[\id_{C^*_r\left(\Gamma_{\sigma_i}\right)}\right]$. Then
$$
\left[\pi_i\right] \bullet \left(\left[\pi_i\right]^{-1} \bullet
  \left[\pi_i\right]\right) = \left(\left[\pi_i\right] \bullet \left[\pi_i\right]^{-1}\right) \bullet
  \left[\pi_i\right] = \left[\pi_i\right]
$$
by associativity of the Kasparov product so the injectivity of
(\ref{eq4}) implies that $\left[\pi_i\right]^{-1} \bullet
  \left[\pi_i\right] =
  \left[\id_{C^*_r\left(\Gamma_{\Sigma}\right)}\right]$, i.e. $\left[\pi_i\right]^{-1}$  is a KK-inverse
of $\left[\pi_i\right]$. To finish the proof, set $\lambda = \left[\pi_1\right] \bullet \left[\pi_0\right]^{-1}$.
\end{proof}
\end{lemma}

Recall that a continuous map $\psi : X \to X$ is \emph{strongly transitive} when
for every open non-empty subset $V$ of $X$, there is an 
$N \in \mathbb N$ such that $\bigcup_{i=0}^N \psi^i(V) = X$.
It was shown in \cite{DS} that when
$\phi : X \to X$ is a surjective local homeomorphism on a compact
metric space $X$, the $C^*$-algebra $C^*_r\left(\Gamma_{\phi}\right)$
is simple if and only if $X$ is not a finite set and $\phi$ is
strongly transitive. In \cite{Th2} and \cite{CT} it was shown that the
$C^*$-algebra of a non-injective and surjective
strongly transitive local homeomorphism on a compact metric space of finite covering
dimension is purely infinite. Combined with Lemma \ref{KKlem} this
leads to the following.

\begin{thm}\label{iso}  Let $X$ be a finite dimensional compact metric
  space and $\varphi : X \to X$, $\phi: X \to X$
  two surjective local homeomorphisms; both non-injective and strongly transitive. Assume that $\phi$ and $\varphi$ are homotopic
  as local homeomorphisms. It follows that
  $C^*_r\left(\Gamma_{\phi}\right) \simeq
  C^*_r\left(\Gamma_{\varphi}\right)$.
\begin{proof} It follows from Corollary 6.6 of \cite{CT} that the
  classification result of Kirchberg and Phillips applies,
  cf. Corollary 4.2.2 of \cite{Ph1}. The conclusion follows therefore
  from Lemma \ref{KKlem}.
\end{proof}
\end{thm}

\subsection{Strong transitivity and exactness}

Let $X$ be a compact metric space which is not a finite set, and $\phi : X \to X$ a continuous
map. Recall that $\phi$ is \emph{exact} when for every open non-empty subset $V \subseteq X$ there is an $N \in \mathbb
N$ such that $\phi^N(U) = X$. Thus exactness implies strong
transitivity while the converse is generally not true. (For example an
irrational rotation of the circle is strongly transitive but not
exact.) It was pointed out in \cite{DS} that a surjective local homeomorphism $\varphi : X \to X$
is exact if and only if $C^*_r\left(R_{\varphi}\right)$ is
simple. Thus $\varphi$ is exact if and only if
$C^*_r\left(\Gamma_{\varphi}\right)$ and
$C^*_r\left(R_{\varphi}\right)$ are both simple while $\varphi$ is strongly
  transitive and not exact if and only
  $C^*_r\left(\Gamma_{\varphi}\right)$ is simple while
$C^*_r\left(R_{\varphi}\right)$ is not.

With this section we want to point out that for locally injective 
and surjective endomorphisms of compact groups, strong transitivity is
equivalent to exactness.



\begin{lemma}\label{equivalent1} Let $\phi : X \to X$ be a continuous,
  surjective
  and open. The following are equivalent:
\begin{enumerate}
\item[i)] $\phi$ is strongly transitive.
\item[ii)] $\bigcup_{n,m \in \mathbb N}
  \phi^{-m}\left(\phi^n(x)\right)$ is dense in $X$ for all $x \in X$.
\item[iii)] $\bigcup_{n \in \mathbb N} \phi^{-n}(x)$ is dense in $X$ for all $x \in X$. 
\end{enumerate}
\begin{proof} i) $\Rightarrow$ iii): If there is point $x \in X$ such that 
$$
F = \overline{ \bigcup_{n \in \mathbb N} \phi^{-n}(x)}
$$
is not all of $X$, the set $U = X \backslash F$ is open, non-empty
and satisfies that $x \notin \bigcup_n \phi^n(U)$, contradicting the
strong transitivity of $\phi$.

ii) $\Rightarrow$ i): Consider an open non-empty subset $V$
of $X$. For every $x \in X$ there are $n,m \in \mathbb N$ such that
$\phi^{-m}\left(\phi^n(x)\right) \cap V \neq \emptyset$, i.e. $x \in
\phi^{-n}\left(\phi^m(V)\right)$. Since $\phi$ is continuous and open, and $X$
compact there is
an $N \in \mathbb N$ such that $X = \bigcup_{i,j \leq N}
\phi^{-i}\left(\phi^j(V)\right)$. Then $X = \phi^N(X) =
\bigcup_{i=0}^{2N}\phi^i(V)$. 

Since iii) $\Rightarrow$ ii) is trivial, the proof is complete. 
\end{proof}
\end{lemma}

\begin{prop}\label{exactendo} Let $H$ be a compact group and
  $\alpha_0 : H \to H$ a continuous surjective group endomorphism with
  finite kernel. Then $\alpha_0$ is exact if and only
  if $\alpha_0$ is strongly transitive.
\begin{proof} Note that $\alpha_0$ is open since its kernel is
  finite. Assume that $\alpha_0$ is strongly transitive, and let $1
  \in H$ be the neutral element. Consider an open non-empty
  subset $U \subseteq H$. Set
$$
\Delta = \bigcup_n \ker \alpha_0^n = \bigcup_n \alpha_0^{-n}(1) .
$$  
Then $\Delta$ is dense in $H$ by Lemma
\ref{equivalent1}. For every $x \in H$,
$$
\bigcup_n \alpha_0^{-n}\left(\alpha_0^n(x)\right) = \left\{ z x : \ z \in
  \Delta \right\},
$$
and it follows that $\bigcup_n \alpha_0^{-n}\left(\alpha_0^n(x)\right)$ is
dense in $H$ for every $x \in H$. In particular, there is
for every $x$ an $n \in \mathbb N$ such that $x \in
\alpha_0^{-n}\left(\alpha_0^n(U)\right)$. Since 
$$
\alpha_0^{-m}\left(\alpha_0^m(U)\right) \subseteq
\alpha_0^{-m-1}\left(\alpha_0^{m+1}(U)\right)
$$
for all $m$ the compactness of $H$ implies that $H
= \alpha_0^{-N}\left(\alpha_0^N(U)\right)$ and therefore that $H =
\alpha_0^N(U)$ for some $N$.
\end{proof}
\end{prop}

\section{The algebra of an affine map on a compact abelian group}


Let $H$ be a compact metrizable abelian group and let $G = \widehat{H}$ be
its Pontryagin dual group. Let $\alpha : H \to H$ be a continuous affine map. That is, $\alpha$
is the composition of a continuous group endomorphism
$\alpha_0 : H \to H$ and the translation by an element $h_0 \in H$, viz.
$$
\alpha(h) = h_0\alpha_0(h).
$$ 
We will refer to $\alpha_0$ as \emph{the linear part} of $\alpha$. To ensure that the transformation groupoid of $\alpha$ is a
well-behaved \'etale groupoid it is necessary to assume that
$\alpha$ is a local homeomorphism.

Let $\left< \cdot, \ \cdot\right>$ denote
the duality between $H$ and $G$. We can then define an endomorphism
$\phi : G \to G$ such that
\begin{equation}\label{dual}
\left<\phi(g),h\right> = \left<g,\alpha_0(h)\right>.
\end{equation}
\begin{lemma}\label{first} The following conditions are equivalent.
\begin{enumerate} 
\item[i)] $\alpha$ is a local homeomorphism.
\item[ii)] $\alpha_0$ is a local homeomorphism.
\item[iii)] $\ker \alpha_0$ and $\coker \alpha_0$ are finite.
\item[iv)] $\ker \phi$ and $\coker \phi$ are finite.
\item[v)] $\ker \alpha_0$ and $\ker \phi$ are finite.
\end{enumerate}
\end{lemma}
\begin{proof} Straightforward.
\end{proof}

Observe that when $H$ is connected $\coker \alpha_0$ is finite if and only if
$\alpha_0$ is surjective. 

Assume that $\alpha : H \to H$ is an affine local homeomorphism. For
each $g \in G$ we define a unitary $U'_g$
in $C(H) \subseteq C^*_r\left(\Gamma_{\alpha}\right)$ in the usual
way: $U'_g(x) = \left< g,x\right>$. Then $U'$ is a representation of
$G$ by unitaries in $C^*_r\left(\Gamma_{\alpha}\right)$. Set $N = \# \ker \alpha_0$ and
define an isometry $V_{\alpha} \in C_c\left(\Gamma_{\alpha}\right)$
such that
$$
V_{\alpha}(x,k,y) = \begin{cases} \frac{1}{\sqrt{N}} & \ \text{when} \ 
  k = 1, \ y = \alpha(x) \\ 0 & \ \text{otherwise.} \end{cases}
$$
It is straightforward to check that $V_{\alpha}U'_g = \left<
  g,h_0\right> U'_{\phi(g)}V_{\alpha}$ and that 
$$
\sum_{g \in
  G/\phi(G)} U'_gV_{\alpha}V_{\alpha}^*{U'_g}^* = 1. 
$$
It follows that we
can consider the universal $C^*$-algebra $\mathcal A[\alpha]$
generated by unitaries $U_g, g \in G$, and an isometry $S$ such that
\begin{equation}\label{relations}
U_gU_h = U_{g+h} \ \ \ \  \ SU_g = \left<g, h_0\right> U_{\phi(g)} S
\ \ \ \  \
\sum_{g \in G/\phi(G)} U_gSS^*U_g^* = 1 .
\end{equation}
Furthermore, there is a $*$-homomorphism $\nu : \mathcal A[\alpha] \to
C^*_r\left(\Gamma_{\alpha}\right)$ such that $\nu(U_g) = U'_g$ and
$\nu(S) = V_{\alpha}$. Note that the existence of $\nu$
implies that the canonical map $C(H) \to \mathcal A[\alpha]$ coming from the generators
$U_g, g \in G$, is injective.

\begin{thm}\label{universal} Assume that $\alpha$ is a local
  homeomorphism. Then $C^*_r\left(\Gamma_{\alpha}\right) \simeq
  \mathcal A[\alpha]$.
\end{thm} 
\begin{proof} To construct the desired isomorphism we will show
  that the isomorphism $\rho : \mathcal O_E \to
  C^*_r\left(\Gamma_{\alpha}\right)$ from Proposition 3.2 in \cite{Th2}
factorises through $\nu$, i.e. that there is a
$*$-homomorphism $\mu : \mathcal O_E \to \mathcal A[\alpha]$ such that
\begin{equation}\label{triangle}
\begin{xymatrix}{
{\mathcal O_E} \ar[rr]^-{\rho} \ar[dr]_-{\mu} & & C^*_r\left(\Gamma_{\alpha}\right) \\
 & {\mathcal A[\alpha]} \ar[ur]_-{\nu} & }
\end{xymatrix}
\end{equation}
commutes. Since $\rho$ is an isomorphism this will complete the proof
if we also show that $\mu$ is surjective. Let $g_i, i = 1,2, \dots, N$, be
elements in $G$ representing the distinct elements of
$G/\phi(G)$. Notice that it follows from the third of the
three relations in (\ref{relations}) that
$$
S^*U_{g_i}^*U_{g_j}S = \begin{cases} 1 \ &
  \text{when $i =j$} \\ 0 \ & \text{when $i \neq j$.} \end{cases}
$$
Combined with the second relation this implies that
$$
S^*U_gS = \begin{cases} 0 & \ \text{when} \ g \notin \phi(G) \\
  \overline{\left< k, h_0\right>} U_{k} & \ \text{when} \ g = \phi(k),
  \ k\in G. \end{cases}
$$
In particular, it follows that the closure of $C(H)S$ in $\mathcal A[\alpha]$ is a Hilbert
$C(H)$-module with the 'inner product' $(a,b) = a^*b$. The existence
of the $*$-homomorphism $\nu$, or a simple direct
calculation shows that
$$
V_{\alpha}^* U'_gV_{\alpha} = \begin{cases} 0 & \ \text{when} \ g \notin \phi(G) \\
  \overline{\left< k, h_0\right>} U'_{k} & \
  \text{when} \ g = \phi(k), \ k \in G. \end{cases} .
$$
Since $E$ is the closure of
$C(H)V_{\alpha}$ in $C^*_r(\Gamma_{\alpha})$ it follows that we can define an isometry $t : E \to
\mathcal A[\alpha]$ such that $t(fV_{\alpha}) = fS$ for all $f \in C(H)$. Together with
the inclusion $\pi : C(H) \to \mathcal A[\alpha]$ this isometry $t$
give us a representation
of the $C^*$-correspondence $E$ in the sense of Katsura, cf. Definition
2.1 of \cite{Ka}. To
show that this representation is covariant in the sense of \cite{Ka} it suffices by Proposition
3.3 in \cite{Ka} to show, in Katsuras notation, that $C(H) \subseteq \psi_t\left(\mathcal
  K(E)\right)$. This follows from the observation that
$$
U_k = \sum_j U_kU_{g_j}SS^*U_{g_j}^* = \sum_{j} \psi_t\left( \theta_{U'_k U'_{g_j} V_{\alpha} , \ U'_{g_j}
    V_{\alpha}} \right)
$$
for all $k$. Thus $(\pi,t)$ is covariant and by Definition 3.5 in \cite{Ka} there is therefore
a $*$-homomorphism $\mu : \mathcal O_E \to \mathcal A[\alpha]$ whose range is generated
by $\pi(C(H))$ and $t(E)$. But this is all of $\mathcal A[\alpha]$, i.e. $\mu$ is surjective.

It remains now only to show that $\rho = \nu \circ \mu$. To this end
observe that the two $*$-homomorphisms agree on the canonical copies
of $C(H)$ and $E$ inside $\mathcal O_E$. As $\mathcal O_E$ is
generated by these subsets the proof is complete.
\end{proof}

For a locally homeomorphic affine map $\alpha$, there is a special
feature of the six-terms exact
sequence of Theorem \ref{6terms} which has been observed already in
\cite{EHR} and \cite{CV} in the case of an endomorphism. Set $N = \#
\ker \alpha_0$ and choose
representatives $g_i, i = 1,2, \dots, N$, in $G$ for the elements of
$G/\phi(G)$. Define $\xi_i \in C_c\left(\Gamma_{\alpha}\right)$ such
that
$$
\xi_i (x,k,y) = \begin{cases} N^{-1/2} \left<g_i,x\right> & \ \text{when} \
  (x,k,y) \in \Gamma_{\alpha}(1,0) \\ 0 & \ \text{otherwise.}
\end{cases}
$$  
Then $\xi_i \in E$ and a straightforward calculation shows that 
$$
\xi_i^*\xi_j = \begin{cases} 1 & \ \text{when} \ i = j \\ 0 & \
  \text{when} \ i \neq j , \end{cases}  
$$
and that $\sum_{i=1}^N \xi_i\xi_i^* = 1$. It follows that $E$ is
isomorphic to $C(H)^N$ as a Hilbert $C(H)$-module under the
isomorphism 
$f \mapsto \left(\xi_1^*f,\xi^*_2f, \dots, \xi^*_Nf\right)$.
It follows that the element 
$[E] \in KK(C(H),C(H))$
 is represented by
the $*$-homomorphism $e : C(H) \to M_N(C(H))$ given by
$e(f) = \left( \xi_i^* f\xi_j\right)_{i,j}$. A straightforward calculation shows that 
$e( f \circ \alpha) = \diag \left(f,f, \dots, f\right)$.
Thus
\begin{equation}\label{Nid}
[E]_* \circ \alpha_* = N \id
\end{equation}
on $K$-theory. In some cases this identity is enough to determine the
action of $[E]$ on $K$-theory.

\subsection{Affine maps with the same linear part} The following result follows immediately from Lemma \ref{KKlem} and Theorem \ref{iso}.

\begin{thm}\label{KK} Let $H$ be a compact abelian path-connected second
  countable group and
  $\alpha_0 : H \to H$ a continuous group endomorphism with
  finite non-trivial kernel. Let $h_i \in H, i = 1,2$, and define $\alpha_i : H \to
  H$ such that $\alpha_i(h) = h_i\alpha_0(h)$.

a) It follows that $
C^*_r\left(\Gamma_{\alpha_1}\right)$ and
$C^*_r\left(\Gamma_{\alpha_2}\right)$ are KK-equivalent.

\smallskip

b) Assume in addition that $\alpha_1$
  and $\alpha_2$ are both strongly transitive, and that $H$ is of
  finite covering dimension. It follows that 
$$
C^*_r\left(\Gamma_{\alpha_1}\right) \simeq C^*_r\left(\Gamma_{\alpha_2}\right) .
$$ 
\end{thm}

\begin{cor}\label{affinecon} Let $H$ be a compact abelian path-connected second
  countable group of finite covering dimension. Let 
  $\alpha : H \to H$ be an affine map whose linear part $\alpha_0$ is a continuous group endomorphism with
  finite non-trivial kernel. Assume that $\alpha$ is exact. Then
  $\alpha_0$ is exact and 
$$
C^*_r\left(\Gamma_{\alpha}\right) \simeq C^*_r\left(\Gamma_{\alpha_0}\right) .
$$ 
\end{cor}
\begin{proof} It is easy to see that $\alpha$ is exact if and only of
  $\alpha_0$ is. Apply then Theorem \ref{KK}.
\end{proof}

Theorem
\ref{KK} a) is generally not true when $H$ is not connected, but it may be that b) of Theorem \ref{KK} and Corollary \ref{affinecon}
remain true also when $H$ is not
connected; at least we do not know of a counterexample.

\begin{example}\label{On} Let $A$ be a finite abelian group of order
  $N$. On the infinite product $A^{\mathbb N}$ the shift $\alpha_0$,
  given by $\alpha_0((a_n)) = (a_{n+1})$, is a surjective exact
  endomorphism with finite kernel. Fix an element $x = (x_n) \in
  A^{\mathbb N}$ and consider the affine map $\alpha : A^{\mathbb N}
  \to A^{\mathbb N}$ defined by 
$$
\alpha((a_n)) = (x_na_{n+1}) .
$$  
The dual group is $\oplus_{k \in \mathbb N} A$ and the dual
endomorphism $\phi$ of $\alpha_0$ is given by
$$
\phi((a_n)) = (0,a_0,a_1,a_2, \dots )
$$
Consider the $C^*$-algebra $\mathcal A[\alpha]$ generated by unitaries
and an isometry satisfying (\ref{relations}).
For $a \in A$, set $g_a = (a,0,0,0, \dots )$. Then $V_a =
U_{g_a}S,  \ a \in A$, is a collection of isometries in $\mathcal
A[\alpha]$ and $\sum_{a \in A} V_aV_a^* = 1$. Since
$$
U_{g_b} = \sum_{a \in A} V_{b+a}V_a^* 
$$
and
$$
U_{\phi^k(g_a)}V_bV_b^*  \in \mathbb C V_b^*U_{\phi^{k-1}(g_a)}V_b
$$
for all $k \geq 1$ and all $a,b$, we conclude that the $V_a$'s
generate $\mathcal A[\alpha]$. It follows that $\mathcal A[\alpha]$ is
a copy of the Cuntz algebra $\mathcal O_N$. In particular, $\mathcal
A[\alpha]$ is independent of the translation part of $\alpha$.
\end{example}

\section{Strongly transitive affine surjections on tori}\label{A}

A continuous map $T : \mathbb T^n \to \mathbb T^n$ on the
$n$-torus is affine when it is the composition of a group
endomorphism $\phi_A : \mathbb T^n \to \mathbb T^n$ and the translation by an
element $\lambda \in \mathbb T^n$, i.e.
$$
Tx = \lambda\phi_A(x).
$$
Being a continuous group endomorphism $\phi_A$ is determined
by an integer matrix $A  = \left(a_{ij}\right) \in M_n(\mathbb Z)$ by the formula
\begin{equation}\label{endoform}
\phi_A\left(t_1,t_2, \dots, t_n\right) = \left(
  t_1^{a_{11}}t_2^{a_{12}}\cdots t_n^{a_{1n}},
  t_1^{a_{21}}t_2^{a_{22}}\cdots t_n^{a_{2n}} , \cdots \cdots,
  t_1^{a_{n1}}t_2^{a_{n2}}\cdots t_n^{a_{nn}} \right)
\end{equation}
for all $\left(t_1,t_2, \dots, t_n\right) \in \mathbb T^n$. It follows
from Lemma \ref{dual} that
$T$ is a local homeomorphism if and only if it is surjective and
finite-to-one. In fact, since surjectivity of $\phi_A$ is equivalent
to non-singularity of $A$, and hence implies that $T$ is finite-to-one
we conclude that $T$ is a local homeomorphism if and only if $A$ is non-singular,
i.e. $\Det A \neq 0$. We call $A$ \emph{the matrix of the linear part} of $T$.

In \cite{Kr} Krzyzewski has given an algebraic characterisation of
which surjective group endomorphisms of tori are strongly transitive and we will
here use his results to obtain a similar characterisation of which
affine surjections are strongly transitive. To formulate Krzyzewski's
result recall that a non-constant polynomial 
$$
a_kx^k + a_{k-1} x^{k-1}
+ \dots + a_0
$$ 
is called
\emph{unimodular} when $a_i \in \mathbb Z$ for all $i$, $a_k = 1$, and
$a_0 \in \{-1,1\}$.  

\begin{thm}\label{Krz} (Krzyzewski, \cite{Kr}) Let $A \in M_n(\mathbb
  Z)$ be non-singular, i.e. $\Det A \neq 0$, and let $f_{A}(x) = \Det (x1 - A)$ be the
  characteristic polynomial of $A$. The group endomorphism $\phi_A$ of
  $\mathbb T^n$ is strongly transitive if and only if no unimodular
  polynomial divides $f_A$.
\end{thm}

\begin{lemma}\label{conjugacy} Let $\phi : \mathbb T^n \to \mathbb
  T^n$ be a surjective affine endomorphism and $A \in M_n(\mathbb Z)$
  the matrix of its linear part. Assume
  that $1$ is not an eigenvalue of $A$. There is
  then a translation $\tau$ on $\mathbb T^n$ such
  that $\tau\phi \tau^{-1} = \phi_A$.
\begin{proof} Let $\lambda \in
  \mathbb R^n$ be a vector such that $\mathbb R^n \ni x \mapsto Ax +
  \lambda$ is a lift of $\phi$, i.e. $\phi(p(x)) = p\left(Ax +
    \lambda\right)$ where $p : \mathbb R^n \to \mathbb T^n$ is the
  canonical surjection. Since $A-1$ is surjective on $\mathbb R^n$ by
  assumption there is a vector
  $\mu \in \mathbb R^n$ such that 
  $\lambda = (A-1)\mu$. Define $\tau$ such that $\tau t = p(\mu)t$ and note that
  $\tau\phi \tau^{-1} = \phi_A$.
\end{proof}
 \end{lemma}

\begin{thm}\label{OK?2} Let $A \in M_n(\mathbb Z)$ be an integral
  matrix with non-zero determinant and let $f_{A}(x) = \Det (x1 - A)$ be the
  characteristic polynomial of $A$. Write
$$
f_A(x) = (1-x)^kg(x)
$$
where $k \in \{0,1,2,\dots, n\}$ is the algebraic multiplicity of $1$
as a root of $f_A$.
\begin{enumerate}
\item[1)] If no unimodular polynomial divides $f_{A}$ every
  affine local homeomorphism of $\mathbb T^n$ with $\phi_A$ as linear
  part is exact and conjugate to $\phi_A$.

\smallskip

\item[2)] If $k \geq 1$ but no unimodular
  polynomial divides $g$, let $S$ be the set of elements $\mu \in
  \mathbb T^n$ with the property that the closed subgroup of $\mathbb
  T^n$ generated by $\mu$ and 
\begin{equation}\label{diffset}
\left\{ x^{-1}\phi_A(x) : \ x \in \mathbb T^n \right\}
\end{equation}
is all of $\mathbb T^n$. Then $S$ is a dense proper subset of $\mathbb
  T^n$ such that an affine map
\begin{equation}\label{affine13}
Tx = \lambda\phi_A(x)
\end{equation}
is strongly transitive if and only if $\lambda \in S$. In
this case no affine local homeomorphism with $\phi_A$ as linear part
is exact.

\smallskip

\item[3)] If there is a unimodular polynomial which divides $g$, no affine 
  local homeomorphism of $\mathbb T^n$ with $\phi_A$ as linear part is
  strongly transitive.
\end{enumerate}
\end{thm}
\begin{proof} 1) In this case $1$ is not an eigenvalue of $A$ and
  hence every affine local homeomorphism with $\phi_A$ as linear part
  is conjugate to $\phi_A$ by Lemma \ref{conjugacy}. It follows from
  Theorem \ref{Krz} and Proposition
  \ref{exactendo} that $\phi_A$ is exact.

2) Note that an affine map is exact if and only if its linear part
is. In the present case it follows from Theorem \ref{Krz} and
Proposition \ref{exactendo} that $\phi_A$ is not
exact; hence no affine map with $\phi_A$ as linear part is exact. This
justifies the last assertion in 2) and shows that the set $S$ is
proper since it does not contain the neutral element. It is dense
because it contains all topological group generators of $\mathbb T^n$. 
To prove the remaining assertions in 2), consider first the
affine map (\ref{affine13}) and assume that it is strongly
transitive. Let $H$ be the closed subgroup of $\mathbb T^n$ generated
by $\lambda$ and (\ref{diffset}). It is easy to check that $T^{-1}(H)
\subseteq H$ and then condition iii) of Lemma \ref{equivalent1} shows
that $H = \mathbb T^n$, i.e. $\lambda \in S$. Conversely, assume that
$\lambda \in S$. We use now Theorem III.12. on page 50 of \cite{N} to
get $W \in Gl_n(\mathbb Z)$ and matrices $B_{11} \in M_{n-k}(\mathbb
Z)$, $B_{22}\in M_{k}(\mathbb Z)$ and a $(n -k)
\times k$ integral matrix $B_{12}$ such that 
\begin{equation}\label{matrixeq2}
W AW^{-1} = \left( \begin{matrix} B_{11} & B_{12} \\ 0 & B_{22} \end{matrix} \right)
\end{equation}
and such that $B_{22}$ is an upper triangular
matrix of the form
$$
B_{22} = \left( \begin{matrix} 1 & b_{12} & b_{13} & \dots & b_{1k} \\
    0 & 1 & b_{23} & \dots & b_{2k} \\
0 & 0 & 1 & \dots & b_{3k} \\
\vdots & \vdots & \vdots & \ddots & \vdots \\ 0 & 0 & 0 & \hdots & 1
\end{matrix} \right)
$$
and $g$ is the characteristic polynomial of $B_{11}$. By
exchanging $\lambda$ with $\phi_W(\lambda)$ and $A$ with $WAW^{-1}$ we may assume that
$A$ is the matrix on the right-hand side of (\ref{matrixeq2}). 
Let $V$ and $U$ be
open non-empty subsets of $\mathbb R^{n-k} $ and $\mathbb R^{k}$,
respectively. It follows from Theorem \ref{Krz} and Proposition \ref{exactendo} that $\phi_{B_{11}}$ is exact which
implies that there is a $N \in \mathbb N$ such that
\begin{equation}\label{jeq7}
B_{11}^{j}V + \mathbb Z^{n-k} = \mathbb R^{n-k} 
\end{equation}
for all $j \geq N$. Let
$\alpha \in \mathbb R^{n-k}, \beta \in \mathbb R^k$ be vectors such
that $p((\alpha,\beta)) = \lambda$ where $p : \mathbb R^n \to \mathbb
T^n$ is the canonical surjection. Let $\rho \in \mathbb T^k$ be the
image of $\beta$ and note that the affine map
$z \mapsto \rho \phi_{B_{22}}(z)$
is an affine homeomorphism of $\mathbb T^k$ which, thanks to the block
diagonal form (\ref{matrixeq2}) is a factor of
$T$. Since $\lambda \in S$ it follows that $\rho$ and the set
$\left\{ z^{-1} \phi_{B_{22}}(z) : \ z \in \mathbb T^k \right\}$
generate $\mathbb T^k$ as a topological group. It follows then
  from Theorem 4 in \cite{HP} that
  $x \mapsto \rho\phi_{B_{22}}(x)$ is a minimal homeomorphism of
  $\mathbb T^k$. There is therefore an $M \in \mathbb
N$ such that
\begin{equation}\label{jeq8}
\bigcup_{j=N}^{N+M} \left(B_{22}^{j}U + \sum_{l=0}^{j-1} B_{22}^{l}\beta\right) + \mathbb Z^k = \mathbb R^k .
\end{equation}
Let $L : \mathbb R^n \to \mathbb R^n$ be the affine map $Lx = Ax +
(\alpha, \beta)$. Thanks to the block form of $A$ it follows from (\ref{jeq7}) and (\ref{jeq8}) that
$$
\bigcup_{j = N}^{N+M} L^j( V \times U)  + \mathbb Z^n = \mathbb R^n,
$$
proving that $x \mapsto \lambda \phi_A(x)$ is strongly transitive.

3) It follows
from Theorem III.12 on page 50 of \cite{N} that there are $W_1
\in Gl_n(\mathbb Z)$, matrices $A_{11} \in M_k(\mathbb Z)$, $A_{22}\in M_{n-k}(\mathbb Z)$ and a $k
\times (n-k)$ integral matrix $A_{12}$ such that
\begin{equation}\label{matrixeq}
W_1 AW_1^{-1} = \left( \begin{matrix} A_{11} & A_{12} \\ 0 & A_{22} \end{matrix} \right)
\end{equation}
and the characteristic polynomials of $A_{11}$ and $A_{22}$ are
$(1-x)^k$ and $g$, respectively. We may therefore assume that $A$ is
the matrix on the right-hand side of (\ref{matrixeq}). Let $T$ be an
affine local homeomorphism with $\phi_A$ as linear
part. The projection $q : \mathbb T^n \to \mathbb T^{n-k}$ to the last
coordinates gives then a factor map to an affine local homeomorphism
$S : \mathbb T^{n-k} \to \mathbb T^{n-k}$ whose linear part in
$\phi_{A_{22}}$. Since $1$ is not an eigenvalue of $A_{22}$ it follows
from Lemma \ref{conjugacy} that $S$ is conjugate to $\phi_{A_{22}}$
which by Theorem \ref{Krz} is not strongly transitive. Since strong
transitivity is inherited by factors it follows that $T$ is not
strongly transitive either.
\end{proof}

\subsection{Local homeomorphisms of the circle}

With this section we want to point out that for the torus of lowest dimension
the group endomorphisms comprise all strongly transitive non-invertible local homeomorphisms, up
to conjugacy.

Let $\mathbb T$ be the unit circle in the complex plane. For any
continuous map $\phi : \mathbb T  \to \mathbb T$ there is a unique
continuous map $g : [0,1] \to \mathbb R$ such that $g(0) \in
[0,1[$ and $\phi\left(e^{2 \pi it} \right) = e^{2 \pi i g(t)}$ for
all $t \in [0,1]$. The value $g(1) - g(0)  \in \mathbb Z$ is the
\emph{degree} of $\phi$ and we denote it by $d_{\phi}$. For maps of positive degree the following can be deduced from the work
of Boyland in \cite{B}. 

\begin{thm}\label{boyland} Let $\phi : \mathbb T  \to \mathbb T$ be a local
  homeomorphism such $\left|d_{\phi}\right| \geq 2$. Assume that
  $\phi$ is strongly transitive. Then $\phi$ is
  conjugate to the endomorphism $z \mapsto z^{d_{\phi}}$.
\end{thm}
\begin{proof}  The proof is essentially the same which is standard for
  expanding maps. Let $p : \mathbb R \to \mathbb T$ be the covering $p(t)
  = e^{2 \pi i t}$ and choose a continuous function $g : \mathbb R \to
  \mathbb R$ such that $\phi \circ p = p \circ g$. Note that $g(x+1) =
  g(x) + d_{\phi}$ and that $g$ is strictly increasing when $d_{\phi}
  \geq 2$ and strictly decreasing when $d_{\phi} \leq -2$. Let $M$ denote the
  set of non-decreasing continuous functions $f : \mathbb R \to \mathbb R$ with the
  property that $f(t+1) = f(t)+1$ for all $t$. Then $M$ is a complete
  metric space in the metric $D$ defined such that
$$
D(f_1,f_2) =  \sup_{t \in \mathbb R}
\left|f_1(t) - f_2(t)\right| .
$$
Define $T_g : M \to M$ such that
$$
T_g(f) = d_{\phi}^{-1} f \circ g .
$$
Then $D\left(T_g(f_1),T_g(f_2)\right) \leq \frac{1}{\left|d_{\phi}\right|}
D(f_1,f_2) \leq 2^{-1}D(f_1,f_2)$, and there is therefore a function
$h \in M$ such that $T_g(h) =h$. Define $\alpha : \mathbb T \to \mathbb
T$ such that $\alpha \circ p = p \circ h$ and observe that $\alpha
\circ \phi = \beta_{d_{\phi}} \circ \alpha$, where
$\beta_{d_{\phi}}(z) = z^{d_{\phi}}$. We claim that $h$ is strictly
increasing. Indeed, if not there is a non-empty open interval in $\mathbb R$ on
which $h$ is constant, and hence also a non-empty open interval $I \subseteq
\mathbb T$ on which $\alpha$ is constant. Since $\phi$ is strongly
transitive there is an $N \in \mathbb N$ such that $\mathbb T =
\bigcup_{j=0}^N \phi^j(I)$. Since $\alpha
\circ \phi^j = \beta_{d_{\phi}}^j \circ \alpha$ for all $j \in \mathbb
N$ it follows that $\alpha$ is constant on $\phi^j(I)$ for all $j$,
whence $\alpha$ is constant  because $\mathbb T$ is
connected. This is impossible since $h \in M$ and hence $h$ is strictly
increasing as claimed. It follows that $\alpha$ is a conjugacy.
\end{proof}

\section{K-theory calculations}
It follows from Lemma
\ref{KKlem} that the $K$-theory groups of
$C^*_r\left(\Gamma_T\right)$ are unchanged when we replace $T$ by its
linear part. We focus therefore in this section on the
calculation of the $K$-groups of $C^*_r\left(\Gamma_{\phi}\right)$
when $\phi$ is a group endomorphism of $\mathbb T^n$. The main tool
will be the six-terms exact sequence from Theorem 3.7 in
\cite{EHR}. Note that it applies to our case since
$C^*_r\left(\Gamma_{\phi}\right)$ is an Exel system in the sense of
\cite{EHR} by Theorem 9.1 of \cite{EV}, and that it is the same
as the Deaconu-Muhly six-terms
exact sequence considered in Section \ref{deamuh}.

Let $\phi_A : \mathbb T^n  \to \mathbb T^n$ be a group endomorphism
given by the integral non-singular matrix $A$, i.e. $\phi_A$ is
defined by (\ref{endoform}). Set $D = \left|\Det A\right|$
and let ${\phi_A}_* : K_*\left(C\left(\mathbb T^n\right)\right) \to
K_*\left(C\left(\mathbb T^n\right)\right), * = 0,1$, be the
homomorphism induced by the endomorphism of $C\left(\mathbb
  T^n\right)$ sending $f$ to $f \circ \phi_A$ and let $\iota :
C(\mathbb T^n) \to C^*_r\left(\Gamma_{\phi_A}\right)$ be the canonical
embedding. It follows from (\ref{Nid}) that the six terms exact sequence from Theorem 3.7 of
\cite{EHR} or Theorem \ref{6terms} takes the form
\begin{equation}\label{eq7}
\begin{xymatrix}{
 K_0\left(C\left(\mathbb T^n\right)\right) \ar[rr]^-{\id -
   D{{\phi_A}_0}^{-1}} & &
 K_0\left(C\left(\mathbb T^n\right)\right) \ar[rr]^-{\iota_*} & &
 K_0\left(C^*_r\left(\Gamma_{\phi_A}\right)\right) \ar[d] \\
K_1\left(C^*_r\left(\Gamma_{\phi_A}\right)\right) \ar[u] & &
K_1\left(C\left(\mathbb T^n\right)\right) \ar[ll]^-{\iota_*} & &
K_1\left(C\left(\mathbb T^n\right)\right)
\ar[ll]^-{\id - D{{\phi_A}_1}^{-1}} 
}\end{xymatrix}
\end{equation}
Consequently
\begin{equation}\label{K0X}
K_0\left(C^*_r\left(\Gamma_{\phi_A}\right)\right) \simeq \coker \left( \id -
   D{{\phi_A}_0}^{-1}\right) \oplus \ker \left( \id -
   D{{\phi_A}_1}^{-1}\right)
\end{equation}
and
\begin{equation}\label{K1X}
K_1\left(C^*_r\left(\Gamma_{\phi_A}\right)\right) \simeq \coker \left( \id -
   D{{\phi_A}_1}^{-1}\right) \oplus \ker \left( \id -
   D{{\phi_A}_0}^{-1}\right) .
\end{equation}
For classification purposes it is important to keep track of the distinguished element of
$K_0\left(C^*_r\left(\Gamma_{\phi_A}\right)\right)$ represented by the
unit in $C^*_r\left(\Gamma_{\phi_A}\right)$. This is always quite easy
because the unit of $C^*_r\left(\Gamma_{\phi_A}\right)$ is the image
of the unit in $C(\mathbb T^n)$ under the embedding $\iota$. In
particular, the unit represents always an element in the direct
summand $\coker \left( \id -
   D{{\phi_A}_0}^{-1}\right)$.

\subsection{The circle}\label{circle} As pointed out in \cite{EHR} the calculation of the
K-theory groups of the $C^*$-algebra of an endomorphism of the circle of
positive degree has
been carried out by several mathematicians, and \cite{EHR} contains
the calculation for endomorphisms of negative degree. We just record the
result. 

Let $a \in \mathbb Z\backslash \{0\}$ and define $\phi_a : \mathbb T \to \mathbb T$
such that $\phi_a(t) = t^a$. Then
\begin{enumerate}
\item[1)] $K_0\left(C^*_r\left(\Gamma_{\phi_a} \right)\right) =
  \mathbb Z_{a-1} \oplus \mathbb Z$ and
  $K_1\left(C^*_r\left(\Gamma_{\phi_a} \right)\right) = \mathbb Z$
  when $a \geq 2$,
\item[2)] $K_0\left(C^*_r\left(\Gamma_{\phi_1} \right)\right) =
  \mathbb Z^2 = K_1\left(C^*_r\left(\Gamma_{\phi_1} \right)\right)$,
\item[3)] $K_0\left(C^*_r\left(\Gamma_{\phi_{-1}} \right)\right) =
  \mathbb Z$ and
  $K_1\left(C^*_r\left(\Gamma_{\phi_{-1}} \right)\right) = \mathbb Z_2$
    and
\item[4)] $K_0\left(C^*_r\left(\Gamma_{\phi_a} \right)\right) =
  \mathbb Z_{|a|-1}$ and
  $K_1\left(C^*_r\left(\Gamma_{\phi_a} \right)\right) = \mathbb Z_2$
  when $a \leq -2$.
 \end{enumerate}

\subsection{The two-torus}\label{n=2}

Before we specialise to $n =2$, observe that by use of the identifications $K_*\left(C\left(\mathbb T^n\right)\right)
\simeq \mathbb Z^{2^{n-1}}, \ * \in \{0,1\}$, the map $A  \mapsto
{\phi_A}_*$ gives rise to maps
\begin{equation}\label{zz1}
\mu_i : M_n \left( \mathbb Z\right) \to M_{2^{n-1}}(\mathbb Z),
\end{equation}
$i = 0,1$. Since $\phi_{AB} = \phi_A \circ \phi_B$ it follows from the
functoriality of $K$-theory that $\mu_0$ and $\mu_1$ are both anti-multiplicative:
\begin{equation}\label{mult}
\mu_i(AB)^t = \mu_i(A)^t\mu_i(B)^t 
\end{equation} 
where $Y^t$ is the transpose of $Y$.
This observation will be a main tool in our quest for a more
detailed description of the K-theory groups in dimension 2 and
3. Other tools will be the following lemmas.


\begin{lemma}\label{kasperslemma} Let $\mu : M_n(\mathbb Z) \to
  \mathbb Z$ be a map such that 
\begin{enumerate}
\item[i)] $\mu(AB) = \mu(A)\mu(B)$ for all $A,B \in M_n(\mathbb Z)$,
  and
\item[ii)] $\mu(A) =  \Det A$ when $A$ is diagonal. 
\end{enumerate}
It follows that $\mu(A) =  \Det A$ for all $A \in M_n(\mathbb Z)$.
\begin{proof} For any $A \in M_n(\mathbb Z)$ there is a diagonal matrix $D
  \in M_n(\mathbb Z)$ and elements $U,V \in Gl_n(\mathbb Z)$ such that
  $A = UDV$; this is the socallled Smith normal form of $A$. It
  suffices therefore to show that $\mu$ agrees with the determinant
  on $Gl_n(\mathbb Z)$, which under the present assumptions amounts to
  showing that $\mu$ is constant 1 on $Sl_n(\mathbb Z)$. When $n \geq
  3$ this is a trivial consequence of conditions i) and ii) given the known
  fact that $Sl_n(\mathbb Z)$ is the commutator subgroup of
  $Gl_n(\mathbb Z)$. This argument does not suffice when $n =2$
  because $Sl_2(\mathbb Z)$ is larger than the commutator subgroup in
  $Gl_2(\mathbb Z)$. To handle the case $n = 2$ recall that
 the two
  matrices $\left(\begin{smallmatrix} 1 & 1 \\ 0 &
      1\end{smallmatrix}\right)$ and $\left(\begin{smallmatrix} 1 & 0 \\ 1 & 1\end{smallmatrix}\right)$
generate $Sl_2(\mathbb Z)$. Since $\mu$ is multiplicative and takes the
identity matrix to 1 the value of $\mu$ on any of the two generators
must be $1$ or $-1$. The fact that the value is $1$ follows from the identity
$$
\left(\begin{smallmatrix} 2 & 1 \\ 0 & 2\end{smallmatrix}\right)^2 =
\left(\begin{smallmatrix} 2 & 0 \\ 0 &
    2\end{smallmatrix}\right)^2\left(\begin{smallmatrix} 1 & 1 \\ 0 &
    1\end{smallmatrix}\right)
$$
and its transpose.
\end{proof}
\end{lemma}

The same proof yields also the following
\begin{lemma}\label{kasperslemma2} Let $\mu : M_n(\mathbb Z) \to
  \mathbb Z$ be a map such that 
\begin{enumerate}
\item[i)] $\mu(AB) = \mu(A)\mu(B)$ for all $A,B \in M_n(\mathbb Z)$,
  and
\item[ii)] $\mu(A) =  1$ when $A$ is diagonal. 
\end{enumerate}
It follows that $\mu(A) =  1$ for all $A \in M_n(\mathbb Z)$.
\end{lemma}

Assume now that $n =2$. From the K\"unneth theorem, cf. \cite{S}, we get isomorphisms
$$
k_0 : \mathbb Z^2 =\left(K_0\left(C\left(\mathbb T\right)\right) \otimes K_0\left(C
  \left( \mathbb T\right) \right)\right) \oplus \left(K_1\left(C\left(\mathbb T\right)\right) \otimes K_1\left(C
  \left( \mathbb T\right) \right)\right) \to K_0\left(C\left(\mathbb
    T^2\right)\right) = \mathbb Z^2
$$
and
$$
k_1 : \mathbb Z^2 =\left( K_1\left(C\left(\mathbb T\right)\right) \otimes K_0\left(C
  \left( \mathbb T\right) \right)\right) \oplus \left(K_0\left(C\left(\mathbb T\right)\right) \otimes K_1\left(C
  \left( \mathbb T\right) \right)\right)  \to K_1\left(C\left(\mathbb
    T^2\right)\right) = \mathbb Z^2.
$$
When we identify $K_*(C(\mathbb T^2))$
with the domain of $k_*$ the naturality of these
isomorphisms implies that
\begin{equation*}\label{diag0}
\mu_0 \left( \begin{smallmatrix} a & 0 \\ 0 & b \end{smallmatrix}
\right) = \left( \begin{smallmatrix} 1 & 0 \\ 0 & ab \end{smallmatrix}
\right)
\end{equation*}
and
\begin{equation*}\label{diag1}
\mu_1 \left( \begin{smallmatrix} a & 0 \\ 0 & b \end{smallmatrix}
\right) = \left( \begin{smallmatrix} a & 0 \\ 0 & b \end{smallmatrix}
\right)
\end{equation*}
for all $a ,b \in \mathbb Z$. Since $\mu_0(A)$ commutes with 
$\mu_0 \left( \begin{smallmatrix} 2 & 0 \\ 0 & 2 \end{smallmatrix}
\right) = \left( \begin{smallmatrix} 1 & 0 \\ 0 & 4 \end{smallmatrix}
\right)$
for all $A \in M_2(\mathbb Z)$, it follows that each $\mu_0(A)$ is
diagonal, i.e. there are multiplicative maps $\delta_i : M_2(\mathbb
Z) \to \mathbb Z, i = 1,2$, such that
$$
\mu_0(A) = \left( \begin{smallmatrix} \delta_1(A) & 0 \\ 0 & \delta_2(A) \end{smallmatrix}
\right) .
$$
It follows then from Lemma \ref{kasperslemma} and Lemma
\ref{kasperslemma2} that
\begin{equation}\label{diag2x}
\mu_0(A) = \left( \begin{smallmatrix} 1 & 0 \\ 0 & \Det A
  \end{smallmatrix} \right) 
\end{equation}
for all $A \in M_2(\mathbb Z)$.

To determine $\mu_1$ we use the following lemma.
\begin{lemma}\label{kasperslemma3} Let $\mu : M_2(\mathbb Z) \to
  M_2(\mathbb Z)$ be a map such that 
\begin{enumerate}
\item[i)] $\mu(AB) = \mu(A)\mu(B)$ for all $A,B \in M_2(\mathbb Z)$,
  and
\item[ii)] $\mu \left( \begin{smallmatrix} a & 0 \\ 0 & b
    \end{smallmatrix} \right) = \left( \begin{smallmatrix} a & 0 \\ 0 & b
    \end{smallmatrix} \right)$ for all $a,b \in \mathbb Z$.
\end{enumerate}
It follows that $\mu(A) =  A$ for all $A \in M_2(\mathbb Z)$ or
$\mu(A) = WAW$ for all $A \in M_2(\mathbb Z)$, where $W = \left(
  \begin{smallmatrix} -1 & 0 \\ 0 & 1 \end{smallmatrix} \right)$ .
\begin{proof} Let $\left( \begin{smallmatrix} a & b \\ c & d
    \end{smallmatrix} \right) \in M_2(\mathbb Z)$ and write
$\mu \left( \begin{smallmatrix} a & b \\ c & d
    \end{smallmatrix} \right) = \left( \begin{smallmatrix} a' & b' \\ c' & d'
    \end{smallmatrix} \right)$.
Since
$\left( \begin{smallmatrix} 1 & 0 \\ 0 & 0
    \end{smallmatrix} \right)\left( \begin{smallmatrix} a & b \\ c & d
    \end{smallmatrix} \right)\left( \begin{smallmatrix} 1 & 0 \\ 0 & 0
    \end{smallmatrix} \right) = \left( \begin{smallmatrix} a & 0 \\ 0 & 0
    \end{smallmatrix} \right)$
it follows from i) and ii) that $a'=a$.
By using $\left( \begin{smallmatrix} 0 & 0 \\ 0 & 1
    \end{smallmatrix} \right)$ in the same way we find also that $b' =
  b$. An application of this conclusion to $\left( \begin{smallmatrix} 0 & 1 \\ 1 & 0
    \end{smallmatrix} \right)$ shows that
$$
\mu \left( \begin{smallmatrix} 0 & 1 \\ 1 & 0
    \end{smallmatrix} \right) = \left( \begin{smallmatrix} 0 & x \\ y & 0
    \end{smallmatrix} \right)
$$
for some $x,y \in \mathbb Z$. Since $\mu \left( \begin{smallmatrix} 0 & 1 \\ 1 & 0
    \end{smallmatrix} \right)^2 = \left( \begin{smallmatrix} 1 & 0 \\ 0 & 1
    \end{smallmatrix} \right)$, we find that $x=y \in \{1,-1\}$. 
After conjugation with $W$, if necessary, we can arrange that $x= y =1$. 
Then
$$
\mu \left( \begin{smallmatrix} 0 & z \\ 0 & 0
    \end{smallmatrix} \right) = \mu \left(\left( \begin{smallmatrix} z & 0
      \\  0 & 0
    \end{smallmatrix} \right) \left( \begin{smallmatrix} 0 & 1 \\ 1  & 0
    \end{smallmatrix} \right)\right) = \left( \begin{smallmatrix} 0 & z \\ 0 & 0
    \end{smallmatrix} \right)
$$
for all $z \in \mathbb Z$. In particular, it follows that
$$
\left( \begin{smallmatrix} 0 & b' \\ 0 & 0 \end{smallmatrix} \right) =
\left( \begin{smallmatrix} 1 & 0 \\ 0 & 0 \end{smallmatrix} \right)
\left( \begin{smallmatrix} a' & b' \\ c' & d' \end{smallmatrix}
\right)\left( \begin{smallmatrix} 0 & 0 \\ 0 & 1 \end{smallmatrix}
\right) = \mu\left(\left( \begin{smallmatrix} 1 & 0 \\ 0 & 0 \end{smallmatrix} \right)
\left( \begin{smallmatrix} a & b \\ c & d \end{smallmatrix}
\right)\left( \begin{smallmatrix} 0 & 0 \\ 0 & 1 \end{smallmatrix}
\right)\right) = \mu\left( \begin{smallmatrix} 0 & b \\ 0 & 0
\end{smallmatrix} \right) = 
\left( \begin{smallmatrix} 0 & b \\ 0 & 0 \end{smallmatrix}
\right),
$$ 
proving that $b' = b$. A similar argument shows that $c' = c$, and
the
proof is complete.
\end{proof}
\end{lemma}

It follows from Lemma \ref{kasperslemma3} that either
\begin{equation*}\label{form1}
\mu_1(A) =  A^t  
\end{equation*}
for all $A \in M_2(\mathbb Z)$, or else
\begin{equation*}
\mu_1(A) =  WA^tW  
\end{equation*}
for all $A \in M_2(\mathbb Z)$, where $W = \left(
  \begin{smallmatrix} -1 & 0 \\ 0 & 1 \end{smallmatrix} \right)$. Either way we can combine with
(\ref{diag2x}) and (\ref{eq7}) to conclude that there is the following six terms exact
sequence, where $\epsilon(A)$ denotes the sign of $\Det A$,
\begin{equation*}\label{eq72}
\begin{xymatrix}{
 \mathbb Z^2 \ar[rr]^-{\left(  \begin{smallmatrix} 1- \left|\Det
         A\right| & 0 \\ 0 & 1 -\epsilon(A) \end{smallmatrix}
\right) } & &
 \mathbb Z^2 \ar[rr] & &
 K_0\left(C^*_r\left(\Gamma_{\phi_A}\right)\right) \ar[d] \\
K_1\left(C^*_r\left(\Gamma_{\phi_A}\right)\right) \ar[u] & &
\mathbb Z^2 \ar[ll] & &
\mathbb Z^2
\ar[ll]^-{1 - \left|\Det A\right| (A^t)^{-1}} 
}\end{xymatrix}
\end{equation*}
Since
$$
\left( \begin{smallmatrix} 0 & 1 \\ -1 & 0 \end{smallmatrix} \right)\left|\Det A\right|(A^t)^{-1} \left( \begin{smallmatrix} 0 & -1
  \\ 1 & 0 \end{smallmatrix} \right) = \epsilon(A) A 
$$
we may exchange $1 - |\Det A|(A^t)^{-1}$ with $1-
\epsilon(A) A$ in the above diagram. In this way we obtain the
following conclusions.

\begin{enumerate}
\item[1)] $K_0\left(C^*_r\left(\Gamma_{\phi_A}\right)\right) \simeq \mathbb Z
\oplus \mathbb Z_{\Det A  -1} \oplus \ker (1 - A)$ with the unit $[1] \in
K_0\left(C^*_r\left(\Gamma_{\phi_A}\right)\right)$ represented by $1
\in \mathbb Z_{\Det A  -1}$ and $K_1\left(C^*_r\left(\Gamma_{\phi_A}\right)\right) \simeq \mathbb Z
\oplus \coker (1-A)$, when $\Det A \geq 2$. 
\item[2)] $K_0\left(C^*_r\left(\Gamma_{\phi_A}\right)\right) \simeq
  \mathbb Z^2 \oplus \ker (1 - A)$ with the unit $[1] \in
K_0\left(C^*_r\left(\Gamma_{\phi_A}\right)\right)$ represented by
$(1,0) 
\in \mathbb Z^2$ and $K_1\left(C^*_r\left(\Gamma_{\phi_A}\right)\right) \simeq \mathbb Z^2
\oplus \coker (1-A)$, when $\Det A =1$.
\item[3)] $K_0\left(C^*_r\left(\Gamma_{\phi_A}\right)\right) \simeq \mathbb Z
\oplus \mathbb Z_2 \oplus \ker (1+ A)$ with the unit $[1] \in
K_0\left(C^*_r\left(\Gamma_{\phi_A}\right)\right)$ represented by $1
\in \mathbb Z$ and $K_1\left(C^*_r\left(\Gamma_{\phi_A}\right)\right) \simeq  \mathbb Z
\oplus \coker (1+A)$,
when $\Det A = -1$.
\item[4)] $K_0\left(C^*_r\left(\Gamma_{\phi_A}\right)\right) \simeq \mathbb Z_2
\oplus \mathbb Z_{|\Det A|  -1} \oplus \ker (1+ A)$  with the unit $[1] \in
K_0\left(C^*_r\left(\Gamma_{\phi_A}\right)\right)$ represented by $1
\in \mathbb Z_{|\Det A|    -1}$ and $
K_1\left(C^*_r\left(\Gamma_{\phi_A}\right)\right) \simeq \coker (1+A)$,
when $\Det A \leq -2$.
\end{enumerate}

When we specialise to the cases where all eigenvalues of $A$ have
modulus greater than $1$ we recover Corollary 4.12 from \cite{EHR}.

\subsection{The three-dimensional torus}\label{n=3} The maps
(\ref{zz1}) we seek are now
$$
\mu_i : M_3(\mathbb Z) \to M_4(\mathbb Z),  \ i = 0,1. 
$$  
To find $\mu_0$ we use the isomorphism $\beta$ from 
\begin{equation*}
\begin{split}
&\left(K_0(C(\mathbb T)) \otimes K_0(C(\mathbb T)) \otimes
  K_0(C(\mathbb T))\right) 
\oplus \left(K_0(C(\mathbb T)) \otimes K_1(C(\mathbb T)) \otimes
  K_1(C(\mathbb T))\right) \\
&\oplus
\left(K_1(C(\mathbb T)) \otimes K_0(C(\mathbb T)) \otimes
  K_1(C(\mathbb T)) \right) \oplus \left(K_1(C(\mathbb T)) \otimes
  K_1(C(\mathbb T)) \otimes K_0(C(\mathbb T))\right)
\end{split}
\end{equation*} 
to $K_0(C(\mathbb T^3))$
given by the K\"unneth theorem, \cite{S}. We will identify
$K_0\left(C\left( \mathbb T^3\right)\right)$ with the domain of $\beta$
and hence with $\mathbb Z^4$ by using the canonical isomorphisms of the summands
with $\mathbb Z$. The naturality of $\beta$ implies that
$$
\mu_0\left( \begin{smallmatrix} a & 0 & 0 \\ 0 & b & 0 \\ 0 & 0 & c
  \end{smallmatrix} \right) = \left( \begin{smallmatrix} 1 & 0 & 0 & 0
    \\ 0 & bc & 0 & 0 \\ 0 & 0 & ac & 0 \\ 0 & 0 & 0 & ab 
  \end{smallmatrix} \right)
$$
for all $a,b,c \in \mathbb Z$. Since $\mu_0(A)$ commutes with 
$$
\mu_0\left( \begin{smallmatrix} 2 & 0 & 0 \\ 0 & 2 & 0 \\ 0 & 0 & 2
  \end{smallmatrix} \right) = \left( \begin{smallmatrix} 1 & 0 & 0 & 0
    \\ 0 & 4 & 0 & 0 \\ 0 & 0 & 4 & 0 \\ 0 & 0 & 0 & 4 
  \end{smallmatrix} \right) 
$$
an application of Lemma \ref{kasperslemma2} shows that there is a map
$\mu'_0 : M_3(\mathbb Z) \to M_3(\mathbb Z)$ such that
$$
\mu_0(A) = \left( \begin{smallmatrix} 1 & 0 \\ 0 & \mu'_0(A) \end{smallmatrix}
\right) .
$$
Note that 
\begin{equation}\label{diag}
\mu'_0\left( \begin{smallmatrix} a & 0 & 0 \\ 0 & b & 0 \\ 0& 0 &
    c\end{smallmatrix} \right) = \left( \begin{smallmatrix} bc & 0 & 0 \\ 0 & ac & 0 \\ 0& 0 &
    ab\end{smallmatrix} \right) 
\end{equation}
for all $a,b,c \in \mathbb Z$. What we need is now the following.

\begin{lemma}\label{id13} Let $\mu : M_3(\mathbb Z) \to M_3(\mathbb Z)$
  be a map such that
\begin{enumerate}
\item[i)] $\mu(AB) = \mu (A) \mu(B)$ for all $A,B \in M_3(\mathbb Z)$
  and
\item[ii)] $\mu \left( \begin{smallmatrix} a & 0 & 0 \\ 0 & b & 0 \\ 0
      & 0 & c \end{smallmatrix}
  \right) =\left( \begin{smallmatrix} bc & 0 & 0 \\ 0 & ac & 0 \\ 0 &
      0 & ab \end{smallmatrix}
  \right)$ for all $a,b,c \in \mathbb Z$. 
\end{enumerate}
After
  conjugation by one of the matrices
\begin{equation}\label{conjumat} \left( \begin{smallmatrix} 1 & 0 & 0 \\ 0 & 1 & 0 \\ 0 & 0 & 1 \end{smallmatrix}
\right), \ 
\left( \begin{smallmatrix} 1 & 0 & 0 \\ 0 & 1 & 0 \\ 0 & 0 & -1 \end{smallmatrix}
\right), \ \left( \begin{smallmatrix} -1 & 0 & 0 \\ 0 & 1 & 0 \\ 0 & 0 & 1 \end{smallmatrix}
\right) \ \text{or} \   \left( \begin{smallmatrix} -1 & 0 & 0 \\ 0 & 1 & 0 \\ 0 & 0 & -1 \end{smallmatrix}
\right),
\end{equation} 
the map $\mu$ agrees with the map which takes
  a matrix to its matrix of cofactors, i.e.
$$ 
 \mu \left( \begin{smallmatrix} a_{11} & a_{12} & a_{13} \\ a_{21} & a_{22}
    & a_{23} \\ a_{31} & a_{32} & a_{33} \end{smallmatrix}
\right) = \left( \begin{matrix} \Det \left( \begin{smallmatrix} a_{22} &
        a_{23} \\ a_{32} & a_{33} \end{smallmatrix} \right) &  - \Det \left( \begin{smallmatrix} a_{21} &
        a_{23} \\ a_{31} & a_{33} \end{smallmatrix} \right) &  \Det \left( \begin{smallmatrix} a_{21} &
        a_{22} \\ a_{31} & a_{32} \end{smallmatrix} \right) \\   -\Det
    \left( \begin{smallmatrix} a_{12} &
        a_{13} \\ a_{32} & a_{33} \end{smallmatrix} \right) &   \Det \left( \begin{smallmatrix} a_{11} &
        a_{13} \\ a_{31} & a_{33} \end{smallmatrix} \right) &  -\Det \left( \begin{smallmatrix} a_{11} &
        a_{12} \\ a_{31} & a_{32} \end{smallmatrix} \right) \\   \Det \left( \begin{smallmatrix} a_{12} &
        a_{13} \\ a_{22} & a_{23} \end{smallmatrix} \right) &   - \Det
    \left( \begin{smallmatrix} a_{11} &
        a_{13} \\ a_{21} & a_{23} \end{smallmatrix} \right) &  \Det \left( \begin{smallmatrix} a_{11} &
        a_{12} \\ a_{21} & a_{22} \end{smallmatrix} \right)
  \end{matrix} \right)
$$
\begin{proof} 
It follows from ii) that
$$
\mu\left( \begin{smallmatrix} 1 & 0 & 0 \\ 0 & 1 & 0 \\ 0 & 0 & 0
  \end{smallmatrix} \right)
  = \left( \begin{smallmatrix} 0 & 0 & 0 \\ 0 & 0 & 0 \\ 0 & 0 & 1
  \end{smallmatrix} \right),
$$
which combined with condition i) implies that 
$\mu\left( \begin{smallmatrix} A & 0 \\ 0 & 1 \end{smallmatrix}\right)$
commutes with $\left( \begin{smallmatrix} 0 & 0 & 0 \\ 0 & 0 & 0 \\ 0 & 0 & 1
  \end{smallmatrix} \right)$ for every $A \in M_2(\mathbb Z)$. It follows
that there are multiplicative maps
$\nu  : M_2(\mathbb Z) \to M_2(\mathbb Z)$
and
$\kappa : M_2(\mathbb Z) \to \mathbb Z$
such that
$$
\mu \left( \begin{smallmatrix} A & 0 \\ 0 & 1 \end{smallmatrix}\right) =
\left( \begin{smallmatrix} \nu(A) & 0 \\ 0 & \kappa(A) \end{smallmatrix}\right)
$$
for all $A \in M_2(\mathbb Z)$. It follows from (\ref{diag}) that
$
\kappa \left( \begin{smallmatrix} a & 0 \\ 0 & b
  \end{smallmatrix}\right) = ab$
for all $a,b \in \mathbb Z$ so Lemma \ref{kasperslemma} implies that
$\kappa(A) = \Det A$. To determine $\nu$ let $a,b,c,d \in \mathbb Z$ and write
$
\nu \left( \begin{smallmatrix} a & b \\ c & d \end{smallmatrix}
\right) = \left( \begin{smallmatrix} a' & b' \\ c' & d' \end{smallmatrix}
\right)$.
Note that $\nu\left( \begin{smallmatrix} a & 0 \\ 0 & 0 \end{smallmatrix}
\right) = \left( \begin{smallmatrix} 0 & 0 \\ 0 & a \end{smallmatrix}
\right)$ by (\ref{diag}). Since
$$
\left( \begin{smallmatrix} a & 0 \\ 0 & 0 \end{smallmatrix}
\right) = \left( \begin{smallmatrix} 1 & 0 \\ 0 & 0 \end{smallmatrix}
\right)\left( \begin{smallmatrix} a & b \\ c & d \end{smallmatrix}
\right)\left( \begin{smallmatrix} 1 & 0 \\ 0 & 0 \end{smallmatrix}
\right),
$$
it follows from i) and (\ref{diag}) that
$d' = a$. By using $\left( \begin{smallmatrix} 0 & 0\\ 0 & 1 \end{smallmatrix}
\right)$ instead it follows in the same way that $a' = d$. In particular, $\nu \left( \begin{smallmatrix} 0 & 1 \\ 1
    & 0 \end{smallmatrix} \right) = \left( \begin{smallmatrix} 0 & x \\ y & 0 \end{smallmatrix}\right)$
for some $x,y \in \mathbb Z$. Since $\nu \left( \begin{smallmatrix} 0 & 1 \\ 1 & 0 \end{smallmatrix}
\right)^2 =  \nu \left( \begin{smallmatrix} 1 & 0 \\ 0 & 1 \end{smallmatrix}
\right) =  \left( \begin{smallmatrix} 1 & 0 \\ 0 & 1 \end{smallmatrix}
\right)$, we find that
\begin{equation}\label{signalt}
\nu \left( \begin{smallmatrix} 0 & 1 \\ 1 & 0 \end{smallmatrix}
\right) =  \left( \begin{smallmatrix} 0 & \pm 1 \\ \pm 1 & 0 \end{smallmatrix}
\right).
\end{equation}
Since 
$\left( \begin{smallmatrix} 0 & 1 \\ 1 & 0 \end{smallmatrix}
\right)\left( \begin{smallmatrix} 1& 0 \\ 0 & 0 \end{smallmatrix}
\right) = \left( \begin{smallmatrix} 0 & 0 \\ 1 & 0 \end{smallmatrix}
\right)$
we conclude that
$$
\nu \left( \begin{smallmatrix} 0 & 0 \\ 1 & 0 \end{smallmatrix}
\right) =  \left( \begin{smallmatrix} 0 & \pm 1 \\ \pm 1 & 0 \end{smallmatrix}
\right) \left( \begin{smallmatrix} 0 & 0 \\ 0 & 1 \end{smallmatrix}
\right) = \left( \begin{smallmatrix} 0 & \pm 1 \\ 0 & 0 \end{smallmatrix}
\right) .
$$
Similarly, we find that
$$
\nu \left( \begin{smallmatrix} 0 & 1 \\ 0 & 0 \end{smallmatrix}
\right) =  \left( \begin{smallmatrix} 0 & 0 \\ \pm 1 & 0 \end{smallmatrix}
\right) .
$$
Note also that $\nu(zA) = z\nu(A)$ for all $z \in \mathbb Z$ and all $A$. We can therefore continue the
quest for the values of $b',c'$ by using that
$$
\left( \begin{smallmatrix} 1 & 0 \\ 0 & 0 \end{smallmatrix}
\right)  \left( \begin{smallmatrix} a & b \\ c & d \end{smallmatrix}
\right)\left( \begin{smallmatrix} 0 & 0 \\ 0 & 1 \end{smallmatrix}
\right) = \left( \begin{smallmatrix} 0 & b \\ 0 & 0 \end{smallmatrix}
\right)
$$
which implies that
$$
b \left( \begin{smallmatrix} 0 & 0 \\ \pm 1 & 0 \end{smallmatrix}
\right) = \nu \left( \begin{smallmatrix} 0 & b \\ 0 & 0 \end{smallmatrix}
\right) = \left( \begin{smallmatrix} 0 & 0 \\ 0 & 1 \end{smallmatrix}
\right)\left( \begin{smallmatrix} a' & b' \\ c' & d' \end{smallmatrix}
\right)\left( \begin{smallmatrix} 1& 0 \\ 0 & 0 \end{smallmatrix}
\right) = \left( \begin{smallmatrix} 0 & 0 \\ c' & 0 \end{smallmatrix}
\right).
$$
Thus $c' = \pm b$. Similarly, we find that $b' = \pm c$, and
hence
$$
\nu\left( \begin{smallmatrix} a & b \\ c & d \end{smallmatrix}
\right) = \left( \begin{smallmatrix} d & \pm c \\ \pm b & a \end{smallmatrix}
\right) .
$$
for all $a,b,c,d \in \mathbb Z$ where the sign only depends
on (\ref{signalt}). 

Now repeat the argument; since $\mu\left(\begin{smallmatrix} 0 & 0 & 0
    \\ 0 & 1 & 0 \\ 0 & 0 & 1 \end{smallmatrix}\right) = \left(\begin{smallmatrix} 1 & 0 & 0
    \\ 0 & 0 & 0 \\ 0 & 0 & 0 \end{smallmatrix}\right)$ a similar
argument shows that
$$
\mu \left( \begin{smallmatrix} 1 & 0 & 0 \\ 0 & a & b \\ 0 & c & d \end{smallmatrix}
\right) = \left(\begin{smallmatrix} ad-cb & 0 & 0
    \\ 0 & d & \pm c \\ 0 & \pm b & a \end{smallmatrix}\right)
$$
where the sign depends on the sign of
$$
\mu \left( \begin{smallmatrix} 1 & 0 & 0 \\ 0 & 0 & 1 \\ 0 & 1 & 0 \end{smallmatrix}
\right) = \mu \left( \begin{smallmatrix} -1 & 0 & 0 \\ 0 & 0 & \pm 1
    \\ 0 & \pm 1 & 0\end{smallmatrix}
\right).
$$
It follows that after conjugation of $\mu$ with one of the matrices
from (\ref{conjumat}) we may assume that
\begin{equation}\label{compare1}
\mu \left( \begin{smallmatrix} 1 & 0 & 0 \\ 0 & a & b \\ 0 & c & d \end{smallmatrix}
\right) = \left(\begin{smallmatrix} ad-cb & 0 & 0
    \\ 0 & d & - c \\ 0 & -b & a \end{smallmatrix}\right)
\end{equation}
for all $a,b,c,d \in \mathbb Z$
and that
\begin{equation}\label{compare2}
\mu \left( \begin{smallmatrix} a & b & 0 \\ c & d & 0 \\ 0 & 0 & 1 \end{smallmatrix}
\right) = \left(\begin{smallmatrix} d & -c & 0
    \\ -b & a & 0 \\ 0 & 0 & ad-cb \end{smallmatrix}\right)
\end{equation}
for all $a,b,c,d \in \mathbb Z$. Note that 
\begin{equation}\label{notethat}
\left( \begin{smallmatrix} 1 & 0 & 0 \\ 0 & 0 & 1 \\ 0 & 1 & 0 \end{smallmatrix}
\right)\left( \begin{smallmatrix} 0 & 1 & 0 \\ 1 & 0 & 0 \\ 0 & 0 & 1 \end{smallmatrix}
\right)\left( \begin{smallmatrix} 1 & 0 & 0 \\ 0 & 0 & 1 \\ 0 & 1 & 0 \end{smallmatrix}
\right) = \left( \begin{smallmatrix} 0 & 0 & 1 \\ 0 & 1 & 0 \\ 1& 0 & 0 \end{smallmatrix}
\right)
\end{equation}
which now implies that
$$
\mu\left( \begin{smallmatrix} 0 & 0 & 1 \\ 0 & 1 & 0 \\ 1& 0 & 0 \end{smallmatrix}
\right) = \left( \begin{smallmatrix} -1 & 0 & 0 \\ 0 & 0 & -1 \\ 0 & -1 & 0 \end{smallmatrix}
\right)\left( \begin{smallmatrix} 0 & -1 & 0 \\ -1 & 0 & 0 \\ 0 & 0 & -1 \end{smallmatrix}
\right)\left( \begin{smallmatrix} -1 & 0 & 0 \\ 0 & 0 & -1 \\ 0 & -1 & 0 \end{smallmatrix}
\right) = \left( \begin{smallmatrix} 0 & 0 & -1 \\ 0 & -1 & 0 \\ -1 & 0 & 0 \end{smallmatrix}
\right).
$$
It follows that when we apply the above arguments to the map
$$
\left( \begin{smallmatrix} a & b  \\ c  & d  \end{smallmatrix}
\right) \mapsto \mu \left( \begin{smallmatrix} a & 0 & b \\ 0 & 1 & 0 \\ c& 0 & d \end{smallmatrix}
\right)
$$
we conclude that
\begin{equation}\label{compare3}
\mu \left( \begin{smallmatrix} a & 0 & b \\ 0 & 1 & 0 \\ c& 0 & d \end{smallmatrix}
\right) = \left( \begin{smallmatrix} d & 0 & -c \\ 0 & ab-cd & 0 \\ -b & 0 & a \end{smallmatrix}
\right)
\end{equation}
for all $a,b,c,d \in \mathbb Z$.

It follows from (\ref{compare1}), (\ref{compare2}) and
  (\ref{compare3}) that $\mu$ agrees with the cofactor map on
  elementary matrices. Since the elementary matrices generate
  $Sl_3(\mathbb Z)$ and since both maps in question are
  multiplicative, we conclude that they agree on $Sl_3(\mathbb
  Z)$. Since the maps also agree on diagonal matrices by (\ref{diag}) it
  follows first that they agree on $Gl_3( \mathbb Z)$ and then from the Smith normal form that they agree on all of
  $M_3(\mathbb Z)$.
\end{proof}
\end{lemma}

If we denote the cofactor matrix of a matrix $A \in M_3(\mathbb Z)$ by
$\cof(A)$ we find that
$$
\mu_0(A) = \left( \begin{smallmatrix} 1 & 0 \\ 0 & W\cof(A)^tW
  \end{smallmatrix} \right) .
$$
where $W$ is one of the matrices from (\ref{conjumat}).

\bigskip

To find $\mu_1$ by the same method we use again the
K\"unneth  theorem which now gives us a natural isomorphism $\alpha$
from 
\begin{equation*}\label{K1}
\begin{split}
&\left(K_1(C(\mathbb T)) \otimes K_1(C(\mathbb T)) \otimes
  K_1(C(\mathbb T))\right) \oplus \left(K_1(C(\mathbb T)) \otimes
  K_0(C(\mathbb T)) \otimes K_0(C(\mathbb T))\right) \oplus \\
&\left(K_0(C(\mathbb T)) \otimes K_1(C(\mathbb T)) \otimes
  K_0(C(\mathbb T))\right) \oplus \left(K_0(C(\mathbb T)) \otimes K_0(C(\mathbb T)) \otimes
  K_1(C(\mathbb T))\right)
\end{split}
\end{equation*}
to
$K_1(C(\mathbb T^3))$.
We will identify $K_1\left(C\left(\mathbb T^3\right)\right)$ with the
domain of $\alpha$ and 
with $\mathbb Z^4$ by using the canonical isomorphisms of each of the
above summands with $\mathbb Z$. The naturality of $\alpha$ implies that the map
$\mu_1 : M_3(\mathbb Z) \to M_4(\mathbb Z)$ which we seek 
satisfies that
\begin{equation}\label{diag2}
\mu_1 \left( \begin{smallmatrix} a & 0 & 0 \\ 0 & b & 0 \\ 0 & 0 & c
  \end{smallmatrix} \right) = \left( \begin{smallmatrix} abc & 0 & 0 & 0\\ 0 & a
    & 0 & 0 \\ 0 & 0 & b & 0 \\ 0 & 0 & 0& c 
  \end{smallmatrix} \right) 
\end{equation}
for all $a,b,c \in \mathbb Z$. Note that $\mu_1(A)$ commutes
with
$$
\mu_1 \left( \begin{smallmatrix} 2 & 0 & 0 \\ 0 & 2 & 0 \\ 0 &
    0 & 2 \end{smallmatrix} \right) = \left( \begin{smallmatrix} 8 & 0 & 0 & 0\\ 0 & 2
    & 0 & 0 \\ 0 & 0 & 2 & 0 \\ 0 & 0 & 0& 2 
  \end{smallmatrix} \right) .
$$
This implies that there are multiplicative maps $\delta : M_3(\mathbb
Z) \to \mathbb Z$ and $\nu : M_3(\mathbb Z) \to  M_3(\mathbb Z)$ such
that
$$
\mu_1(A) = \left(\begin{smallmatrix} \delta(A) & 0 \\ 0 & \nu(A)
  \end{smallmatrix} \right) 
$$
for all $A \in M_3(\mathbb Z)$. By Lemma \ref{kasperslemma}
$\delta(A) = \Det A$. To determine $\nu$ we need the following lemma.

\begin{lemma}\label{id3} Let $\mu : M_3(\mathbb Z) \to M_3(\mathbb Z)$
  be a map such that
\begin{enumerate}
\item[i)] $\mu(AB) = \mu (A) \mu(B)$ for all $A,B \in M_3(\mathbb Z)$
  and
\item[ii)] $\mu \left( \begin{smallmatrix} a & 0 & 0 \\ 0 & b & 0 \\ 0
      & 0 & c \end{smallmatrix}
  \right) =\left( \begin{smallmatrix} a & 0 & 0 \\ 0 & b & 0 \\ 0 &
      0 & c \end{smallmatrix}
  \right)$ for all $a,b,c \in \mathbb Z$ 
\end{enumerate}
It follows that $\mu(A) = WAW$ for all $A \in M_3(\mathbb Z)$, where
$W$ is one of the matrices from (\ref{conjumat}).
\end{lemma}
\begin{proof} It follows from Lemma \ref{kasperslemma3} applied to
$$
M_2(\mathbb Z) \ni A \mapsto \mu \left( \begin{smallmatrix} A & 0 \\ 0
    & 1 \end{smallmatrix} \right)
$$
and
$$
M_2(\mathbb Z) \ni A \mapsto \mu \left( \begin{smallmatrix} 1 & 0 \\ 0
    & A \end{smallmatrix} \right)
$$
that we can arrange, if necessary by conjugating $\mu$ with one of the matrices
from (\ref{conjumat}), that
$$
\mu \left( \begin{smallmatrix} A & 0 \\ 0
    & 1 \end{smallmatrix} \right) =  \left( \begin{smallmatrix} A & 0 \\ 0
    & 1 \end{smallmatrix} \right) \ \text{and} \  \mu \left( \begin{smallmatrix} 1 & 0 \\ 0
    & A \end{smallmatrix} \right) =  \left( \begin{smallmatrix} 1 & 0 \\ 0
    & A \end{smallmatrix} \right)
$$
for all $A \in M_2(\mathbb Z)$. Another application of Lemma
\ref{kasperslemma3} shows that
\begin{equation}\label{pmchoice}
\mu \left( \begin{smallmatrix} a & 0 & b \\ 0 & 1 & 0 \\ c & 0 & d
  \end{smallmatrix} \right) = \left( \begin{smallmatrix} a & 0 & \pm b
    \\ 0 & 1 & 0 \\ \pm c & 0 & d
  \end{smallmatrix} \right) 
\end{equation}
for all $a,b,c,d \in \mathbb Z$. To determine the sign note that it
follows from (\ref{notethat})
that
\begin{equation}\label{notethat2}
\mu \left( \begin{smallmatrix} 0 & 0 & 1 \\ 0 & 1 & 0 \\ 1 & 0 & 0
  \end{smallmatrix} \right) = \left( \begin{smallmatrix} 0 & 0 & 1 \\ 0 & 1 & 0 \\ 1 & 0 & 0
  \end{smallmatrix} \right) ,
\end{equation}
which means the sign in (\ref{pmchoice}) must be $+$. It follows that $\mu$ is the identity on all elementary matrices and hence on all
of $Sl_3(\mathbb Z)$. The Smith normal form then shows that $\mu$ is
the identity on all of $M_3(\mathbb Z)$.
\end{proof}

It follows that up to conjugation by a diagonal matrix
from $Gl_4(\mathbb Z)$, $\mu_1$ has the form
$$
\mu_1(A) = \left( \begin{smallmatrix} \Det A & 0 \\ 0 & A^t \end{smallmatrix}
\right) .
$$
Since $\cof(A) = \Det A \left(A^{-1}\right)^t$ we conclude that the
sixterms exact sequence (\ref{eq7}) in the case $n
=3$ becomes
\begin{equation*}\label{eq73}
\begin{xymatrix}{
 \mathbb Z^4 \ar[rrr]^-{\left(  \begin{smallmatrix} 1- \left|\Det
         A\right| & 0 \\ 0 & 1 - \epsilon(A) A \end{smallmatrix}
\right) } & & &
 \mathbb Z^4 \ar[rrr] & & &
 K_0\left(C^*_r\left(\Gamma_{\phi_A}\right)\right) \ar[d] \\
K_1\left(C^*_r\left(\Gamma_{\phi_A}\right)\right) \ar[u] & & &
\mathbb Z^4 \ar[lll] & & &
\mathbb Z^4 
\ar[lll]^-{\left(  \begin{smallmatrix} 1- \epsilon(A) & 0 \\ 0 & 1 -\epsilon(A)\cof(A) \end{smallmatrix}
\right)} 
}\end{xymatrix}
\end{equation*}
It follows that
\begin{enumerate}
\item[1)] $K_0\left(C^*_r\left(\Gamma_{\phi_A}\right)\right) \simeq \mathbb Z
 \oplus \ker (1-\cof(A)) \oplus \mathbb Z_{\Det A -1} \oplus \coker
 (1-A)$ with the unit $[1] \in
K_0\left(C^*_r\left(\Gamma_{\phi_A}\right)\right)$ represented by $1
\in \mathbb Z_{\Det A  -1}$ and $K_1\left(C^*_r\left(\Gamma_{\phi_A}\right)\right) \simeq \mathbb Z
\oplus \ker (1-A) \oplus \coker \left(1 - \cof(A)\right)$, when $\Det A \geq 2$.
\item[2)] $K_0\left(C^*_r\left(\Gamma_{\phi_A}\right)\right) \simeq
K_1\left(C^*_r\left(\Gamma_{\phi_A}\right)\right) \simeq \mathbb Z^2
 \oplus \ker (1-A) \oplus \coker (1 -A)$ with the unit $[1] \in
K_0\left(C^*_r\left(\Gamma_{\phi_A}\right)\right)$ represented by
$(1,0) 
\in \mathbb Z^2$, when $\Det A = 1$. 
\item[3)] $ K_0\left(C^*_r\left(\Gamma_{\phi_A}\right)\right) \simeq
  \mathbb Z \oplus \ker (1-A) \oplus \coker (1+ A)$ with the unit $[1] \in
K_0\left(C^*_r\left(\Gamma_{\phi_A}\right)\right)$ represented by $1
\in \mathbb Z$ and $K_1\left(C^*_r\left(\Gamma_{\phi_A}\right)\right) \simeq \mathbb Z
\oplus \ker (1+A)
\oplus \coker \left(1 +\cof(A)\right) \oplus \mathbb Z_2$,
when $\Det A = -1$.
\item[4)] $K_0\left(C^*_r\left(\Gamma_{\phi_A}\right)\right) \simeq
  \ker (1+\cof(A)) \oplus \mathbb Z_{|\Det A| -1} \oplus \coker (1+
  A)$ with the unit $[1] \in
K_0\left(C^*_r\left(\Gamma_{\phi_A}\right)\right)$ represented by $1
\in \mathbb Z_{|\Det A|  -1}$ and $K_1\left(C^*_r\left(\Gamma_{\phi_A}\right)\right) \simeq \ker (1+A)
\oplus \coker (1+ \cof(A)) \oplus \mathbb Z_2$, when $\Det \leq -2$.
\end{enumerate}

\section{The $C^*$-algebras of strongly transitive affine surjections
  on an $n$-torus, $n \leq 3$}

\subsection{The $C^*$-algebra of a strongly transitive local
  homeomorphism on the circle}\label{n=1x}

A continuous affine map $T$ of the circle has the form
$$
Tt = e^{2\pi i \alpha} t^a
$$
for some $\alpha \in \mathbb R$ and some $a \in \mathbb Z$. By
combining the results of Section \ref{A} with the
K-theory calculations listed in Section \ref{circle} we obtain the
following conclusion.

\begin{itemize}
\item[A)] When $a \geq 2$ the $C^*$-algebra $C^*_r\left(\Gamma_T\right)$ is the
  same for all $\alpha \in \mathbb R$, it is purely
  infinite and simple with $K$-theory groups
  $$
K_0\left( C^*_r(\Gamma_T)\right) \simeq \mathbb Z \oplus \mathbb
Z_{a-1}, \ \ K_1\left( C^*_r(\Gamma_T)\right) \simeq \mathbb Z.
$$
The unit of $C^*_r(\Gamma_T)$ corresponds to $1 \in \mathbb Z_{a-1}
\subseteq K_0\left( C^*_r(\Gamma_T)\right)$.

\item[B)] When $a \leq -2$
 the $C^*$-algebra $C^*_r(\Gamma_T)$ is the
  same for all $\alpha \in \mathbb R$, it is purely
  infinite and simple with $K$-theory groups
$$
K_0\left( C^*_r(\Gamma_T)\right) \simeq \mathbb
Z_{|a|-1}, \ \ K_1\left( C^*_r(\Gamma_T)\right) \simeq \mathbb Z_2.
$$
The unit of $C^*_r(\Gamma_T)$ corresponds to $1 \in \mathbb Z_{|a|-1}
\subseteq K_0\left( C^*_r(\Gamma_T)\right)$.

\item[C)] When $a = \pm 1$, $T$ is a
homeomorphism and 
$C^*_r\left(\Gamma_T\right)$ is finite. When $a = -1$, $T$ is not
strongly transitive and $C^*_r\left(\Gamma_T\right)$ is not
simple. When $a = 1$, $T$ is strongly transitive if and only if
$\alpha$ is not rational. When $\alpha \in \mathbb R \backslash
\mathbb Q$, $C^*_r\left(\Gamma_T\right)$ is an irrational rotation
algebra and its structure is well-known. See \cite{EE}. 
\end{itemize}

It is well-known that two irrational rotation algebras are isomorphic if
and only if the two irrational rotations are conjugate. Now combine
this with the observation that the
degree $a$ of $T$ can be read off from the K-theory groups of
$C^*_r(\Gamma_T)$, and the
well-known fact that a minimal homeomorphism of the circle is
conjugate to an irrational rotation. Combining with Theorem
\ref{boyland} we obtain then the following result regarding strongly transitive local homeomorphisms of the
circle.

\begin{prop}\label{ciclec} Two strongly transitive local homeomorphisms
  $\varphi$ and $\psi$ of the circle are conjugate if and only if the
  associated $C^*$-algebras $C^*_r\left(\Gamma_{\varphi}\right)$ and
  $C^*_r\left(\Gamma_{\psi}\right)$ are isomorphic.
\end{prop}

\subsection{The $C^*$-algebra of a strongly transitive affine
  surjection on the two-torus}\label{alg-two}

An affine local homeomorphism of $\mathbb T^2$
has the form
\begin{equation}\label{formn=2}
Tx = \lambda\phi_A(x)
\end{equation}
for some $A \in M_2(\mathbb Z)$ with $\Det
  A \neq 0$ and some $\lambda \in \mathbb T^2$.

Most of the following results summarise the results of Theorem
\ref{OK?2}, Corollary \ref{affinecon} and
  Section \ref{n=2}, but the case $\Det A = 1$ uses the
  calculation of N.C. Phillips from Example 4.9 of \cite{Ph2} and the
  classification results of Lin and Phillips from \cite{LP}.

Assuming that neither $1$ nor $-1$ is an eigenvalue of $A$ we have the
following:
\begin{itemize}
\item[A)] When $\Det A \geq 2$ the $C^*$-algebra $C^*_r(\Gamma_T)$ is
  isomorphic to $C^*_r\left(\Gamma_{\phi_A}\right)$, it is purely
  infinite and simple with $K$-theory groups
$$
K_0\left( C^*_r(\Gamma_T)\right) \simeq \mathbb Z \oplus \mathbb
Z_{\Det A -1}, \ \ K_1\left( C^*_r(\Gamma_T)\right) \simeq \mathbb Z
\oplus \coker (1-A).
$$
The unit of $C^*_r(\Gamma_T)$ corresponds to $1 \in \mathbb Z_{\Det A-1}
\subseteq K_0\left( C^*_r(\Gamma_T)\right)$.

\item[B)] When $\Det A \in \{-1,1\}$, $T$ is not strongly transitive
  and $C^*_r(\Gamma_T)$ is not simple.

\item[C)] When $\Det A \leq -2$ the $C^*$-algebra $C^*_r(\Gamma_T)$ is is
  isomorphic to $C^*_r\left(\Gamma_{\phi_A}\right)$, it is purely
  infinite and simple with $K$-theory groups
$$
K_0\left( C^*_r(\Gamma_T)\right) \simeq \mathbb Z_2 \oplus \mathbb
Z_{|\Det A| -1} , \ \ K_1\left( C^*_r(\Gamma_T)\right) \simeq  \coker (1+A).
$$
The unit of $C^*_r(\Gamma_T)$ corresponds to $1 \in \mathbb Z_{|\Det A|-1}
\subseteq K_0\left( C^*_r(\Gamma_T)\right)$.  
\end{itemize}

\smallskip

Note that it follows from Theorem \ref{OK?2} that $T$ is not strongly
transitive and 
  $C^*_r\left(\Gamma_T\right)$ not simple when $-1$ is an eigenvalue
  while $1$ is not. 

\smallskip

 Assuming that $1$ is an eigenvalue of $A$ we have the following: 
\begin{itemize}
\item[D)] When $\Det A \geq 2$ the set of $\lambda$'s for which $T$ is
  strongly transitive is the dense proper subset of $\mathbb T^2$
  consisting of the elements $\lambda \in \mathbb T^2$ with the property that the
  closed group generated by
  $\lambda$ and the set $\left\{ x^{-1}\phi_A(x): \ x \in \mathbb T^2
  \right\}$ is all of $\mathbb T^2$. The
  corresponding $C^*$-algebras $C^*_r(\Gamma_T)$ are the same simple and
  purely infinite $C^*$-algebra with $K$-theory groups
$$
K_0\left( C^*_r(\Gamma_T)\right) \simeq \mathbb Z^2 \oplus \mathbb
Z_{\Det A -1} , \ \ K_1\left( C^*_r(\Gamma_T)\right) \simeq \mathbb Z
\oplus \coker(1-A).
$$
The unit of $C^*_r(\Gamma_T)$ corresponds to $1 \in \mathbb Z_{\Det A-1}
\subseteq K_0\left( C^*_r(\Gamma_T)\right)$.

\item[E)] When $\Det A = 1$, the set of $\lambda$'s for which $T$ is
  strongly transitive is a dense proper subset of $\mathbb T^2$. For
  each such $\lambda$ the $C^*$-algebra $C^*_r\left(\Gamma_T\right)$
  is a simple unital AH-algebra with no dimension growth, a unique trace state and real
  rank zero, cf. Example 5.6 of \cite{LP}. The ordered K-theory groups depend on $\lambda$ and are
  calculated in Example 4.9 of \cite{Ph2}.

\item[F)] When $\Det A = -1$, $T$ is not strongly transitive and $C^*_r\left(\Gamma_T\right)$ is not simple.

\item[G)] When $\Det A \leq -2$ the set of $\lambda$'s for which $T$ is
  strongly transitive is the dense proper subset of $\mathbb T^2$ which
  consists of the elements $\lambda \in \mathbb T^2$ with the property that the
  closed group generated by
  $\lambda$ and the set $\left\{ x^{-1}\phi_A(x): \ x \in \mathbb T^2
  \right\}$ is all of $\mathbb T^2$. The
  corresponding $C^*$-algebras $C^*_r(\Gamma_T)$ are the same
  simple and 
  purely infinite $C^*$-algebra with $K$-theory groups
$$
K_0\left( C^*_r(\Gamma_T)\right) \simeq \mathbb Z_2 \oplus \mathbb
Z_{|\Det A| -1}, \ \ K_1\left( C^*_r(\Gamma_T)\right) \simeq  \coker (1+A).
$$
The unit of $C^*_r(\Gamma_T)$ corresponds to $1 \in \mathbb Z_{|\Det A|-1}
\subseteq K_0\left( C^*_r(\Gamma_T)\right)$.  
\end{itemize}


\subsection{The $C^*$-algebra of a strongly transitive affine
  surjection on the three-dimensional torus}\label{alg-3}

We consider now an affine map $T : \mathbb T^3 \to \mathbb T^3$ of the
form
\begin{equation}\label{T3}
Tx = \lambda \phi_A(x)
 \end{equation}
where $A \in M_3(\mathbb Z)$ and $\Det A \neq 0$.

Assume first that none of the numbers $1, -1, \Det
  A$ and $- \Det A$ are an eigenvalue of $A$.
\begin{itemize}
\item[A)] When $\Det A \geq 2$ the $C^*$-algebra $C^*_r(\Gamma_T)$ is is
  isomorphic to $C^*_r\left(\Gamma_{\phi_A}\right)$, it is purely
  infinite and simple with $K$-theory groups
\begin{equation*}
\begin{split}
& K_0\left(C^*_r\left(\Gamma_{T}\right)\right) \simeq \mathbb Z
\oplus \mathbb Z_{\Det A -1} \oplus \coker (1 -A), \\
&K_1\left(C^*_r\left(\Gamma_{T}\right)\right) \simeq \mathbb Z
\oplus \coker (1 -\cof(A)) .
\end{split}
\end{equation*}
The unit of $C^*_r(\Gamma_T)$ corresponds to $1 \in \mathbb Z_{\Det A-1}
\subseteq K_0\left( C^*_r(\Gamma_T)\right)$.

\item[B)] When $\Det A \in \{-1,1\}$, $T$ is not strongly transitive
  and $C^*_r(\Gamma_T)$ is not simple.

\smallskip

\item[C)] When $\Det A \leq -2$ the $C^*$-algebra $C^*_r(\Gamma_T)$ is
  isomorphic to $C^*_r\left(\Gamma_{\phi_A}\right)$, it is purely
  infinite and simple with $K$-theory groups
\begin{equation*}
\begin{split}
& K_0\left(C^*_r\left(\Gamma_{T}\right)\right) \simeq  \mathbb
Z_{|\Det A| -1} \oplus \coker (1+A), \\
&K_1\left(C^*_r\left(\Gamma_{T}\right)\right) \simeq \mathbb Z_2
\oplus \coker (1 + \cof(A))
\end{split}
\end{equation*}
The unit of $C^*_r(\Gamma_T)$ corresponds to $1 \in \mathbb Z_{|\Det A|-1}
\subseteq K_0\left( C^*_r(\Gamma_T)\right)$.  
\end{itemize}

\smallskip

It follows from Theorem \ref{OK?2} that $T$ is not strongly transitive
when one of the numbers $-1, \ \Det A$
  and $-\Det A$ is an eigenvalue of $A$, but $1$ is
  not, and when both $1$ and $-1$ are eigenvalues of $A$. It remains
  therefore only to consider the case when $1$ is an eigenvalue, but
  $-1$ is not.

\smallskip

Assume that $1$ is an eigenvalue of $A$ and
  that $-1$ is not.
\begin{itemize}
\item[D)] When $\Det A \geq 2$ the set of $\lambda$'s
  for which $T$ is strongly transitive is the dense proper subset of $\mathbb T^2$ which
  consists of $\lambda \in \mathbb T^3$ with the property that the
  closed group generated by
  $\lambda$ and the set $\left\{ x^{-1}\phi_A(x): \ x \in \mathbb T^3
  \right\}$ is all of $\mathbb T^3$. The corresponding $C^*$-algebras
  $C^*_r\left(\Gamma_T\right)$ are all the same simple and purely
  infinite $C^*$-algebra with $K$-theory groups
\begin{equation*}
\begin{split}
& K_0\left(C^*_r\left(\Gamma_{T}\right)\right) \simeq \mathbb Z
 \oplus \ker (1-\cof(A)) \oplus \mathbb Z_{\Det A -1} \oplus \coker (1
 -A), \\
&K_1\left(C^*_r\left(\Gamma_{T}\right)\right) \simeq \mathbb Z
\oplus \ker (1-A) \oplus \coker (1 -\cof(A)) .
\end{split}
\end{equation*}
The unit of $C^*_r(\Gamma_T)$ corresponds to $1 \in \mathbb Z_{\Det A-1}
\subseteq K_0\left( C^*_r(\Gamma_T)\right)$.

\item[E)] When $\Det A = 1$, $T$ is not strongly transitive unless $1$
  is the only eigenvalue of $A$ in which case the set of $\lambda$'s
  for which $T$ is strongly transitive is the dense proper subset of $\mathbb T^3$ which
  consists of the $\lambda$'s in $\mathbb T^3$ with the property that the
  closed group generated by
  $\lambda$ and the set $\left\{ x^{-1}\phi_A(x): \ x \in \mathbb T^3
  \right\}$ is all of $\mathbb T^3$. With $\lambda$ in this set, the
  $C^*$-algebra $C^*_r\left(\Gamma_T\right)$ is a unital simple
  AH-algebra with no dimension growth, a unique trace state and real
  rank zero. The ordered $K$-theory of $C^*_r\left(\Gamma_T\right)$
  depends on $\lambda$ and is calculated in Subsection \ref{homeoK} below.

\item[F)] When $\Det A = -1$, $T$ is not strongly transitive and $C^*_r\left(\Gamma_T\right)$ is not simple.

\item[G)] When $\Det A \leq -2$ the set of $\lambda$'s
  for which $T$ is strongly transitive is the dense proper subset of $\mathbb T^2$ which
  consists of $\lambda \in \mathbb T^3$ with the property that the
  closed group generated by
  $\lambda$ and the set $\left\{ x^{-1}\phi_A(x): \ x \in \mathbb T^3
  \right\}$ is all of $\mathbb T^3$. The corresponding $C^*$-algebras
  $C^*_r\left(\Gamma_T\right)$ are all the same simple and purely
  infinite $C^*$-algebra with $K$-theory groups
\begin{equation*}
\begin{split}
& K_0\left(C^*_r\left(\Gamma_{T}\right)\right) \simeq \ker
(1 +\cof(A)) \oplus \mathbb Z_{|\Det A| -1} \oplus \coker (1 +A), \\
&K_1\left(C^*_r\left(\Gamma_{T}\right)\right) \simeq  \mathbb Z_2
\oplus \coker (1 +\cof(A))
\end{split}
\end{equation*}
The unit of $C^*_r(\Gamma_T)$ corresponds to $1 \in \mathbb Z_{|\Det A|-1}
\subseteq K_0\left( C^*_r(\Gamma_T)\right)$.  
\end{itemize}

\subsection{Minimal affine homeomorphisms of the three-dimensional
  torus and their $C^*$-algebras}\label{homeoK}

In this section we justify the statements made under E) in the preceding
section concerning the $C^*$-algebras
of a minimal affine homeomorphism of $\mathbb T^3$. As we shall show
the conclusion concerning the structure of the algebras will follow
from results of Lin and Phillips from \cite{LP} once we have
calculated the K-groups and the action of the traces on $K_0$. To do
this we use the method of Phillips from Example 4.9 of \cite{Ph2}. First of all we note that a minimal affine homeomorphism of a torus is
uniquely ergodic with the Haar measure as the unique invariant Borel
probability measure. This follows from Theorem 4 of \cite{Pa} and it
means that there is only a single trace state to consider. 

Let $A \in M_3(\mathbb Z)$ be a matrix for which $1$ is the only
eigenvalue and let $\lambda = (\lambda_1,\lambda_2,\lambda_3) \in
\mathbb T^3$ be an element such that 
$Tx = \lambda \phi_A(x)$
is minimal. Since $1$ is the only eigenvalue of $A$ minimality of $T$
is equivalent to the condition that the closed group generated by $\lambda$ and the set $\left\{ x^{-1}\phi_A(x): \ x \in \mathbb T^3
  \right\}$ is all of $\mathbb T^3$. This follows from Theorem
  \ref{OK?2}, but in the present case this is actually a result of Hoare
  and Parry, cf. Theorem 4 in \cite{HP}.   

It follows from Theorem III.12 on page 50 of \cite{N} that there an element $W
\in Gl_3(\mathbb Z)$ and integers $a,b,c \in \mathbb Z$ such that
$$
WAW^{-1} = \left( \begin{smallmatrix} 1 & a & b \\ 0 & 1 & c \\ 0 & 0
    & 1 \end{smallmatrix} \right) .
$$
We will therefore assume that $A$ is equal to the matrix on the
right-hand side.

Let $\alpha_T : C(\mathbb T^3) \to C(\mathbb T^3)$ be the automorphism
$\alpha_T(f) = f \circ T$ so that $C^*_r\left(\Gamma_T\right)$ is
isomorphic to the crossed product $C(\mathbb T^3) \times_{\alpha_T}
\mathbb Z$, cf. Proposition 1.8 of \cite{Ph3}. Let $\tau$ be the trace
state of $C(\mathbb T^3) \times_{\alpha_T}
\mathbb Z$ induced by the Haar-measure of $\mathbb T^3$. Thanks to the
unique ergodicity of $T$ this is the only trace state of $C(\mathbb T^3) \times_{\alpha_T}
\mathbb Z$. We aim to calculate the map
$$
\tau_* : K_0\left(C(\mathbb T^3) \times_{\alpha_T}
\mathbb Z\right) \to \mathbb R .
$$

From the six-terms exact sequence of Pimsner and
Voiculescu, \cite{PV}, we consider the piece
\begin{equation}\label{PV}
\begin{xymatrix}{
K_0(C(\mathbb
T^3)) \ar[rr]^-{\id - {{\alpha_T}_*}^{-1}} &&
K_0(C(\mathbb
T^3)) \ar[rr]^-{i_*}  &&  K_0(C_r^*\left(\Gamma_T\right))
\ar[d]^-{\partial} \\
  && K_1\left( C(\mathbb T^3) \right)  && K_1\left( C(\mathbb T^3)
  \right) \ar[ll]^-{\id -
  {{\alpha_T}_*}^{-1}}  }
\end{xymatrix}
\end{equation}
which gives an isomorphism
\begin{equation}\label{K0iso}
K_0(C_r^*\left(\Gamma_T\right)) \simeq \coker \left(\id -
  {{\alpha_T}_*}^{-1}\right) \oplus \ker \left(\id -
  {{\alpha_T}_*}^{-1}\right) .
\end{equation}
From the calculations in Section \ref{n=3} we deduce that 
$\coker \left(\id -
  {{\alpha_T}_*}^{-1}\right) \simeq \mathbb Z \oplus \coker (1- A)$ and
\begin{equation}\label{K1sub}
\ker \left(\id -
  {{\alpha_T}_*}^{-1}\right) \simeq  \mathbb Z \oplus \ker (1-(A^t)^{-1}) .
\end{equation}
On $K_0(C(\mathbb
T^3))$ all traces induce the map which takes a projection to its rank
and it follows therefore that $\tau_* : \mathbb Z \oplus \coker (1-A)
\to \mathbb R$ is the map which picks up the first coordinate from
$\mathbb Z$ and annihilates $\coker (1-A)$.

To see what $\tau_*$ does on $\ker \left(\id -
  {{\alpha_T}_*}^{-1}\right) = \mathbb Z \oplus \ker (1-(A^t)^{-1})$
we need representatives of the 4 generators coming from the direct
summands in (\ref{K1}). The first direct summand $\mathbb Z$ is generated, as a subgroup of $K_1\left(C(\mathbb T^3)\right)$ by the image of $[z]
\otimes [z] \otimes [z] \in K_1\left(C(\mathbb T)\right) \otimes
K_1\left(C(\mathbb T)\right) \otimes K_1\left(C(\mathbb T)\right)$ under the composed map
$$
K_1\left(C(\mathbb T)\right) \otimes
K_1\left(C(\mathbb T)\right) \otimes K_1\left(C(\mathbb T)\right) \to K_1\left(C(\mathbb T)\right) \otimes
K_0\left(C(\mathbb T^2)\right) \to K_1\left(C(\mathbb T^3)\right)
$$
coming from two applications of the K\"unneth theorem. (Here $z$
denotes the identity function on $\mathbb T$, considered as a unitary in
$C(\mathbb T)$.) It follows that
a generator of this subgroup is represented by the product $UU_0$ of two unitaries $U,U_0 \in 
M_2\left(C(\mathbb T^3)\right)$ given by the formula's   
$$
U(z_1,z_2,z_3) = z_1p(z_2,z_3) + (1-p(z_2,z_3)),
$$
$$
U_0(z_1,z_2,z_3) = z_1^{-1}p_0(z_2,z_3) + (1-p_0(z_2,z_3)),
$$
where $p,p_0 \in M_2(C(\mathbb T^2))$ are the rank one projections
described on page 55 of \cite{Ex}. Since $\Det UU_0(z_1,z_2,z_3) = 1$
for all $(z_1,z_2,z_3) \in \mathbb T^3$ it follows from Theorem V. 12 and Theorem
VI. 11 in \cite{Ex} that $\tau_*$ takes this generator $u$, now considered as
an element of $K_0(C_r^*\left(\Gamma_T\right))$, to an integer $k \in
\mathbb Z$. By exchanging $u - k[1]$ for $u$ we can therefore assume that
$\tau_*$ annihilates the first $\mathbb Z$-summand from $\ker \left(\id -
  {{\alpha_T}_*}^{-1}\right)$.

To get a picture of how $\tau_*$ acts on the second summand observe that the generators of the last three summands of (\ref{K1}) are
represented in $K_1\left(C(\mathbb T^3)\right)$ by the unitaries
$u_1,u_2,u_3 \in C(\mathbb T^3)$ defined such that $u_i(z_1,z_2,z_3) =
z_i$. Note that
$$
\ker (1-(A^t)^{-1}) = \left\{ (x_1,x_2,x_3) \in \mathbb Z^3 : \ ax_1  =
  cx_2 + (b-ac)x_1 = 0 \right\} .
$$
The identification of $\ker (1-(A^t)^{-1})$ with a subgroup of
$K_1(C(\mathbb T^3))$ coming from (\ref{K1sub}) will be suppressed in the following; it is  given by the map
$(x_1,x_2,x_3) \mapsto \sum_{i=1}^3 x_i[u_i]$.
Fix a group embedding $\Phi_0 : \ker (1-(A^t)^{-1}) \to
K_0(C_r^*\left(\Gamma_T\right))$ such that $\partial \circ \Phi_0 =
\id$.
Write $\lambda_j = e^{2 \pi i \alpha_j}$ for some $\alpha_j \in
\mathbb R, j = 1,2,3$. It follows then from Theorem IX. 11 of
\cite{Ex} that for any element $\xi = (x_1,x_2,x_3) $ of
$\ker (1-(A^t)^{-1})$ there is an integer $k_{\xi}$ such
that
$$
\tau_* \circ \Phi_0(\xi) = x_1 \alpha_1 + x_2 \alpha_2 + x_3 \alpha_3
+ k_{\xi} .
$$
We can therefore change $\Phi_0$ on group generators to obtain another embedding $\Phi : \ker (1-(A^t)^{-1}) \to
K_0(C_r^*\left(\Gamma_T\right))$ such that $\Phi_0(\xi) - \Phi(\xi)
\in \mathbb Z[1]$ for all $\xi \in \ker (1-(A^t)^{-1})$ and 
$$
\tau_* \circ \Phi(x_1,x_2,x_3) = \sum_{j=1}^3 x_j\alpha_j 
$$
for all $x_1,x_2,x_3 \in \ker (1-(A^t)^{-1})$. It follows that there is an isomorphism 
$$ 
\Psi : \mathbb Z^2 \oplus \ker (1-(A^t)^{-1}) \oplus \coker (1 -A) \to
K_0(C_r^*\left(\Gamma_T\right))
$$
such that $\tau_* \circ \Psi (x,y,u,v) = x + \eta(u)$, where $x,y \in
\mathbb Z, \ u \in \ker (1-(A^t)^{-1}), \ v \in \coker (1-A)$, and 
$\eta : \ker (1-(A^t)^{-1}) \to \mathbb R$
is given by
$$
\eta(x_1,x_2,x_3) = x_1 \alpha_1 + x_2 \alpha_2 + x_3 \alpha_3 .
$$

In all cases $\alpha_3 \in
\tau_*\left(K_0(C_r^*\left(\Gamma_T\right))\right)$. Since $\lambda$
and $\left\{ x^{-1}\phi_A(x) : \ x \in \mathbb T^3 \right\}$ must
generate $\mathbb T^3$ in order for $T$ to be minimal we see that
$\alpha_3$ must be irrational. It follows therefore from Corollary 5.3
of \cite{LP} that $C^*_r(\Gamma_T)$ is a unital simple
  AH-algebra with no dimension growth, a unique trace state and real
  rank zero, as claimed in E) of Section \ref{alg-3}. 

It follows then from \cite{Ph2} that the positive semi-group of
$K_0(C_r^*\left(\Gamma_T\right))$ under the isomorphism $\Psi$ becomes
the set
$$
\{0\} \cup \left\{ (x,y,u,v) \in \mathbb Z^2 \oplus \ker
  (1-(A^t)^{-1}) \oplus \coker (1 -A): \ x + \eta(u) > 0 \right\}.
$$
In this way we have obtained a complete description of
$K_0(C_r^*\left(\Gamma_T\right))$ as a partially ordered group. In
this picture the
order unit coming from the unit in $C_r^*\left(\Gamma_T\right)$ is $(1,0,0,0)$.

The
calculation of the $K_1$-group is very easy in comparison. As in
Section \ref{n=3} we find that
$$
 K_1\left(C^*_r\left(\Gamma_T\right)\right) \simeq
 K_0\left(C^*_r\left(\Gamma_T\right)\right) \simeq \mathbb Z^2 \oplus
 \ker (1-A) \oplus \coker (1-A). 
$$

\begin{remark} When only one of the $\lambda_i$'s are different from 1
  the preceding description of the order on $K_0$ can be obtained from
  Theorem 7.2 of \cite{Rei}. On the other hand it follows from the
  calculation above that in our setting the range of the trace on
  $K_0$ can have rank 3 and 4 which is not possible when the
  $C^*$-algebra comes from a Furstenberg transformation. 
\end{remark}


\begin{thebibliography}{WWW} 













\bibitem[An]{An} C. Anantharaman-Delaroche, {\em Purely infinite $C^*$-algebras arising from dynamical systems}, Bull. Soc. Math. France {\bf 125} (1997), 199--225.































\bibitem[Bl]{Bl} B. Blackadar, {\em K-theory for Operator Algebras}, Springer Verlag, New York, 1986.






\bibitem[B]{B} P. Boyland, {\em Semiconjugacies to angle-doubling},
  Proc. Amer. Math. Soc. {\bf 134} (2006), 1299-1307. 










 

























\bibitem[BGR]{BGR} L. Brown, P. Green, M. Rieffel, {\em Stable isomorphism and strong Morita equivalence of $C^*$-algebras}, Pacific J. Math. {\bf 71} (1977), 349--363.





\bibitem[CT]{CT} T.M. Carlsen and K. Thomsen, {\em The structure of
    the $C^*$-algebra of a locally injective surjection}, Ergod. Th. \& Dynam. Sys., to appear.


















\bibitem[CV]{CV} J. Cuntz and A. Vershik, {\em $C^*$-algebras
    associated with endomorphisms and polymorphisms of compact abelian
    groups}, arXiv:1202.5960.


\bibitem[De]{De} V. Deaconu, {\em Groupoids associated with endomorphisms}, Trans. Amer. Math. Soc. {\bf 347} (1995), 1779-1786.


\bibitem[DS]{DS} V. Deaconu and F. Schultz, {\em $C^*$-algebras
    associated with interval maps}, Trans. Amer. Math. Soc. {\bf 359}
  (2007), 1889-1924.




\bibitem[DM]{DM} V. Deaconu and P.S. Muhly, {\em $C^*$-algebras
    associated with branched coverings}, Proc. Amer. Math. Soc. {\bf
    129} (2000), 1077-1086.







 



 


















\bibitem[EE]{EE} G. Elliott and D. Evans, {\em The structure of the irrational rotation $C^*$-algebra}, Ann. of Math. (2) {\bf 138} (1993), 477--501.









\bibitem[Ex]{Ex} R. Exel, {\em Rotation numbers for automorphisms of
    $C^*$-algebras}, Pac. J. Math. {\bf 127} (1987), 31-89.


\bibitem[EHR]{EHR} R. Exel, A. an Huef and I. Raeburn, {\em Purely
    infinite simple $C^*$-algebras associated to integer dilation
    matrices}, Indiana Univ. Math. J., to appear. arXiv:1003.2097.


\bibitem[EV]{EV} R. Exel and A. Vershik, {\em $C^*$-algebras of
    Irreversible Dynamical Systems}, Canad. J. Math. {\bf 58} (2006), 39-63.












 
















 























\bibitem[HP]{HP} A. H. M. Hoare and W. Parry, {\em Affine
    transformations with quasi-discrete spectrum, I}, J. London
  Math. Soc. {\bf 41} (1966), 88-96.









 



 







\bibitem[Ka]{Ka} T. Katsura, {\em On $C^*$-algebras associated with
    $C^*$-correspondences}, J. Func. Analysis {\bf 217} (2004), 366-401.
XS





\bibitem[KT]{KT} A. Kishimoto and H. Takai, {\em Some Remarks on
    $C^*$-Dynamical Systems with a Compact Abelian Group}, Publ. RIMS
  {\bf 14} (1978), 383-397. 
















 

\bibitem[Kr]{Kr} K. Krzyzewski, {\em On Exact Toral Endomorphisms},
  Mh. Math.{\bf 116} (1993), 39-47.



 






\bibitem[LP]{LP} H. Lin and N.C. Phillips, {\em Crossed products by
    minimal homeomorphisms}, J. reine angew. Math. {\bf 641} (2010),
  95-122.













 











\bibitem[N]{N} M. Newman, {\em Integral Matrices}, Academic Press, New
  York, London, 1972.




























\bibitem[Pa]{Pa} W. Parry, {\em Ergodic Properties of Affine
    Homeomorphisms and Flows of Nilmanifolds}, Amer. J. Math {\bf 91}
  (1969), 757-771.







\bibitem[Pe]{Pe} G. K. Pedersen, {\em $C^*$-algebras and Their Automorphism Groups}, London Mathematical Society Monographs, Vol. 14 (London: Academic Press, 1979). 


\bibitem[Ph1]{Ph1} N. C. Phillips, {\em A classification theorem for nuclear
    purely infinite simple $C^*$-algebras},
  Doc. Math. {\bf 5} (2000), 49-114.


\bibitem[Ph2]{Ph2} \bysame, {\em Cancellation and stable rank for
    direct limits of recursive subhomogeneous $C^*$-algebras},
  Trans. Amer. Math. Soc. {\bf 359} (2007), 4625-4652.


 
\bibitem[Ph3]{Ph3} \bysame, {\em Crossed Products of the Cantor Set by Free Minimal Actions of $\mathbb Z^d$}, Comm. Math. Phys. {\bf 256} (2005), 1--42.






\bibitem[PV]{PV} M. Pimsner and D. Voiculescu, {\em Exact sequences for $K$-groups and $\Ext$-groups of certain cross-products of $C^*$-algebras}, J. Operator Theory {\bf 4} (1980), 93--118.















\bibitem[RW]{RW} I. Raeburn and D. P. Williams, {\em Morita
    Equivalence and Continuous-Trace $C^*$-algebras}, American
  Mathematical Society, 1998.




\bibitem[Rei]{Rei} K. Reihani, {\em K-theory of the Furstenberg
    transformation group $C^*$-algebra}, arXiv:1109.4473


\bibitem[Re]{Re} J. Renault, {\em A Groupoid Approach to $C^*$-algebras},  LNM 793, Springer Verlag, Berlin, Heidelberg, New York, 1980.  








 








\bibitem[RS]{RS} J. Rosenberg and C. Schochet, {\em The K\"unneth theorem and the universal coefficient theorem for Kasparov's generalized K-functor}, Duke J. Math. {\bf 55} (1987), 337-347.










\bibitem[S]{S} C. Schochet, {\em Topological methods for
    $C^*$-algebras: geometric resolutions and the K\"unneth formula},
  Pac. J. Math. {\bf 98} (1982), 443-458.  






 




\bibitem[Th1]{Th1} K. Thomsen, {\em Semi-\'etale groupoids and
  applications}, Annales de l'Institute Fourier {\bf 60} (2010), 759-800.

\bibitem[Th2]{Th2} K. Thomsen, {\em Pure finiteness
 of the crossed product of an AH-algebra by an endomorphism},
Can. J. Math., to appear. ArXiv:1010.0960.

\bibitem[Th3]{Th3} K. Thomsen, {\em The $C^*$-algebra of the
    exponential function}, Proc. Amer. Math. Soc., to appear. ArXiv:1109.5031. 


































 








 













\end{thebibliography}
\end{document}